\documentclass[reqno,11pt]{amsart}
\usepackage{amsmath,amsfonts,amssymb,amsxtra,latexsym,amscd,enumerate,amsthm,verbatim}

\usepackage{graphicx}

\usepackage{mathtools}
\usepackage{dsfont}
\usepackage{mathrsfs}
\usepackage{color}
\usepackage{bm}
\usepackage{enumitem}

\usepackage{scalerel,stackengine}
\stackMath
\newcommand\reallywidehat[1]{%
\savestack{\tmpbox}{\stretchto{%
  \scaleto{%
    \scalerel*[\widthof{\ensuremath{#1}}]{\kern-.7pt\bigwedge\kern-.7pt}%
    {\rule[-\textheight/2]{1.5ex}{\textheight}}
  }{\textheight}%
}{0.9ex}}%
\stackon[2.25pt]{#1}{\tmpbox}%
}

\usepackage[top=30mm, bottom=30mm, left=30mm, right=30mm]{geometry}
\numberwithin{equation}{section}

\newcommand{\R}{\mathbb{R}}

\newcommand{\M}{\mathcal{M}}

\newcommand{\A}{\mathcal{A}}
\newcommand{\B}{\mathcal{B}}

\newcommand{\G}{\mathbb{G}}
\newcommand{\HH}{\mathbb{H}}
\newcommand{\PP}{\mathbb{P}}
\newcommand{\JJ}{\mathbb{J}}
\newcommand{\T}{\mathbb{T}}

\newcommand{\C}{\mathbb{C}}
\newcommand{\Z}{\mathbb{Z}}
\newcommand{\N}{\mathbb{N}}

\newcommand{\de}{\delta}

\newcommand{\varep}{\varepsilon}
\newcommand{\maj}{\mathrm{maj}}
\newcommand{\minor}{\mathrm{min}}

\numberwithin{equation}{section}

\newcommand{\ex}{\mathfrak{e}}

\newcommand{\low}{\rm low}
\newcommand{\ls}{\lesssim}
\newcommand{\NN}{\mathbb{N}}
\newcommand{\ind}[1]{\mathds{1}_{{#1}}}
\newcommand*{\DMO}[1]{\expandafter\DeclareMathOperator\csname #1\endcsname {#1}}
\DeclarePairedDelimiter\abs{\lvert}{\rvert}

\DeclarePairedDelimiterX\Set[2]{\{}{\}}{#1\colon #2}

\DMO{dyad}

\newtheorem{theorem}{Theorem}[section]
\newtheorem{lemma}[theorem]{Lemma}
\newtheorem{proposition}[theorem]{Proposition}

\newtheorem{remark}[theorem]{Remark}
\newtheorem{conjecture}[theorem]{Conjecture}

\begin{document}

\title[Polynomial sequences in discrete nilpotent groups of step 2]{Polynomial sequences in discrete nilpotent\\ groups of step 2}

\author{Alexandru D. Ionescu}
\address{Princeton University}
\email{aionescu@math.princeton.edu}

\author{{\'A}kos Magyar}
\address{University of Georgia -- Athens}
\email{magyar@math.uga.edu}

\author{Mariusz  Mirek }
\address{Institute for Advanced Study, Princeton (USA) \&
Rutgers University (USA)
\&
Instytut Matematyczny, Uniwersytet Wroc{\l}awski (Poland)}
\email{mariusz.mirek@rutgers.edu}

\author{Tomasz Z. Szarek}
\address{BCAM - Basque Center for Applied Mathematics (Spain)
\&
Instytut Matematyczny, Uniwersytet Wroc{\l}awski (Poland)}
\email{tzszarek@bcamath.org}

\begin{abstract}
We discuss some of our work on averages along polynomial sequences in nilpotent groups of step 2. Our main results include boundedness of associated maximal functions and singular integrals operators, an almost everywhere pointwise convergence theorem for ergodic averages along polynomial sequences, and a nilpotent Waring theorem. 

Our proofs are based on analytical tools, such as a nilpotent Weyl inequality, and on complex almost-orthogonality arguments that are designed to replace Fourier transform tools, which are not available in the non-commutative nilpotent setting.
In particular, we present what we call a \textit{nilpotent circle method} that allows us to adapt some of the ideas of the classical circle method to the setting of nilpotent groups. 
\end{abstract}

\thanks {The first, second and third authors were supported in part by NSF
grants DMS-2007008 and DMS-1600840 and  DMS-2154712 respectively.  The third author 
was also partially supported by the Department of Mathematics at Rutgers
University and by the National Science Centre in Poland, grant Opus
2018/31/B/ST1/00204. The fourth author was partially supported by the
National Science Centre of Poland, grant Opus 2017/27/B/ST1/01623, 
the Juan de la Cierva Incorporaci{\'o}n 2019, grant number IJC2019-039661-I,  
the 
Agencia Estatal de Investigaci{\'o}n, grant PID2020-113156GB-I00/AEI/10.13039/501100011033,
the
Basque Government through the BERC 2022-2025 program,
and by the Spanish Ministry of Sciences, Innovation and
Universities: BCAM Severo Ochoa accreditation SEV-2017-0718.
}

\maketitle

\setcounter{tocdepth}{1}

\begin{center}
\large\textit{Dedicated to David Jerison, on the occasion of his 70th birthday.}    
\end{center}

\tableofcontents

\section{Introduction}

The goal of this paper is twofold. We first review some recent results on averages of functions along polynomial sequences in discrete nilpotent Lie groups of step 2, and the main ideas in the proofs. Then we use one of the main ingredients, a nilpotent Weyl inequality, to prove a new theorem on a nilpotent version of the Waring problem.

The natural general setting for our analysis consists of a {\it{discrete nilpotent group $\G$ of step $d$}}, which by definition is assumed to be a discrete, co-compact subgroup of a connected and simply connected nilpotent Lie group $\G^\#$ of step $d$, and a {\it{polynomial sequence}} $A:\Z\to\G$, which is a map satisfying $A(0)=1$ and $D^{k_0}A\equiv 1$ for some $k_0\geq 1$. Here $D^k$ is the $k$-fold differencing operator defined recursively by
\begin{equation*}
D^0A(n):=A(n),\qquad D^{k+1}A(n):=D^kA(n)^{-1}D^kA(n+1),\qquad n\in\mathbb{Z}.
\end{equation*}
We consider a class of operators defined by taking averages along polynomial sequences in discrete nilpotent groups. As in the continuous case, one can consider discrete maximal operators, which have applications to pointwise ergodic theorems, and discrete Calder{\'o}n--Zygmund operators. 

\subsection{The main theorem} Our main theorem in this paper concerns $L^p$ boundedness of ma\-ximal averages along polynomial sequences in discrete nilpotent groups of step 2, $L^p$ pointwise ergodic theorems, and $L^2$ boundedness of singular integrals. More precisely:

\begin{theorem}[Main result]\label{thm:main}
Assume that $\G$ is a discrete nilpotent group $\G$ of step $2$ and $A:\Z\to\G$ is a polynomial sequence. Then:

(i) ($\ell^p$ boundedness of maximal averages) Assume $f:\mathbb{G}\to\mathbb{C}$ is a function and let
\begin{equation*}
\mathcal{M}f(g):=\sup_{N\geq 0}\frac{1}{2N+1}\sum_{|n|\leq N}|f(A^{-1}(n)\cdot g)|,\qquad g\in\mathbb{G}.
\end{equation*}
Then, for any $p\in(1,\infty]$, 
\begin{equation*}
\|\mathcal{M}f\|_{\ell^p(\mathbb{G})}\lesssim_p \|f\|_{\ell^p(\mathbb{G})}.
\end{equation*}

(ii) ($L^p$ pointwise ergodic theorems) Assume $\mathbb{G}$ acts by measure-preserving transformations on a $\sigma$-finite measure space $X$, $f\in L^p(X)$, $p\in (1,\infty)$, and let
\begin{equation}\label{eq:40}
A_Nf(x):=\frac{1}{2N+1}\sum_{|n|\leq N}f(A^{-1}(n)\cdot x),\qquad x\in X.
\end{equation}
Then the sequence $A_Nf$ converges pointwise almost everywhere and in the $L^p$ norm as $N\to\infty$.

(iii) ($\ell^2$ boundedness of singular averages) Assume $K:\mathbb{R}\to\mathbb{R}$ is a Calder{\'o}n--Zygmund kernel, i.e. a $C^1$ function satisfying
\begin{equation}\label{CalZyg}
\sup_{t\in\mathbb{R}}[(1+|t|)|K(t)|+(1+|t|)^2|K'(t)|]\leq 1,\qquad\sup_{N\geq 0}\Big|\int_{-N}^NK(t)\,dt\Big|\leq 1.
\end{equation} 
Assume that $f:\mathbb{G}\to\mathbb{C}$ is a (compactly supported) function, and let
\begin{equation*}
Hf(g):=\sum_{n\in\mathbb{Z}} K(n)f(A^{-1}(n)\cdot g),\qquad g\in\mathbb{G}.
\end{equation*}
Then
\begin{equation*}
\|Hf\|_{\ell^2(\mathbb{G})}\lesssim \|f\|_{\ell^2(\mathbb{G})}.
\end{equation*}
\end{theorem}

The theorem follows by combining the main results in \cite{IMMS} for parts (i) and (ii), and \cite{IMW} for part (iii). We discuss now some connections between this theorem and other related results in the literature.

\subsubsection{Continuous Radon transforms} The discrete maximal averages and the discrete singular averages defined in Theorem \ref{thm:main} can be thought of as discrete analogues of the continuous Radon transforms, which are averages along suitable curves or surfaces. The theory of continuous Radon transforms has been extensively studied, motivated mainly by problems at the interface of Fourier analysis and geometry of surfaces in Euclidean spaces or nilpotent groups, and is very well understood. This includes $L^q$ estimates for the full range of exponents $q>1$ and multidimensional averages, see for example \cite{Ch}, \cite{RiSt}, \cite{ChNaStWa}.

\subsubsection{The Furstenberg--Bergelson--Leibman conjecture} Discrete averages, both of the maximal and singular type, have been considered motivated mainly by open problems in ergodic theory. A fundamental problem in ergodic theory is to establish convergence in norm and pointwise almost everywhere for the polynomial ergodic averages as in \eqref{eq:40} as $N\to\infty$ for functions $f\in L^p(X)$, $1\le p\le \infty$.
The problem goes back to at least the early 1930's with von Neumann's mean ergodic theorem \cite{vN} and Birkhoff's pointwise ergodic theorem \cite{BI} and led to profound extensions such as Bourgain's polynomial pointwise ergodic theorem  \cite{Bo2, Bo3, Bo1} and Furstenberg's ergodic proof \cite{Fur0} of Szemer{\'e}di's theorem \cite{Sem1} in particular. Furstenberg's proof was also the starting point of ergodic Ramsey theory, which resulted in many natural generalizations
of Szemer{\'e}di's theorem, including a polynomial Szemer{\'e}di theorem of Bergelson and Leibman \cite{BL1}.

This motivates the following far reaching conjecture known as the Furstenberg--Bergelson--Leibman conjecture {\cite[Section 5.5, p. 468]{BeLe}}.

\begin{conjecture}\label{con:1} Assume that $d,k\geq 1$ are integers, $(X, \mathcal B(X), \mu)$ is a probability space, and assume that $T_1,\ldots, T_d:X\to X$ is a given family of invertible measure-preserving transformations on the space $(X, \mathcal B(X), \mu)$ that generates a nilpotent group of step $k$. Assume that $m\geq 1$ is an integer and 
 $P_{1, 1},\ldots,P_{i, j},\ldots, P_{d, m}:\Z\to\Z$ are polynomial maps with integer coefficients such that $P_{i, j} (0) = 0$. Then for any  $f_1, \ldots, f_m\in L^{\infty}(X)$, the non-conventional multilinear 
polynomial averages
\begin{align}
\label{eq:41}
A_{N; X, T_1,\ldots, T_d}^{P_{1, 1}, \ldots, P_{d, m}}(f_1,\ldots, f_m)(x)=\frac{1}{2N+1}\sum_{n\in[-N, N]\cap\Z}\prod_{j=1}^mf_j(T_1^{P_{1, j}(n)}\cdots T_d^{P_{d, j}(n)} x)
\end{align}
converge for $\mu$-almost every $x\in X$ as $N\to\infty$. 
\end{conjecture}

Conjecture \ref{con:1} is a major open problem in ergodic theory that was promoted in person by Furstenberg, see \cite[p. 6662]{A1}, before being published in \cite{BeLe}. Our main result Theorem \ref{thm:main} (ii) proves this conjecture in the linear case $m=1$, provided that the family of transformations $T_1,\ldots, T_d:X\to X$ generates a nilpotent group of step $k=2$.

\subsubsection{Earlier pointwise ergodic theorems} The basic linear case $m=d=k=1$ with $P_{1,1}(n)=n$ follows from  Birkhoff's original ergodic theorem \cite{BI}. On the other hand, the commutative case $m=d=k=1$ with an arbitrary polynomial $P=P_{1,1}$ with integer coefficients was a famous open problem of Bellow \cite{Bel} and Furstenberg \cite{Fur3}, solved by Bourgain in his breakthrough papers \cite{Bo2, Bo3, Bo1}.

Some particular examples of averages \eqref{eq:41} with $m=1$ and polynomial mappings with degree at most two  in the step two nilpotent setting were studied in \cite{IMSW, MSW0}. 

The multilinear theory $m\geq 2$, in contrast to the linear theory, is widely open even in the commutative case $k=1$.  Only a few results in the bilinear  $m=2$ and  commutative  $d=k=1$ setting are known. Bourgain \cite{B0} proved pointwise convergence when $P_{1,1}(n)=an$ and $P_{1,2}(n)=bn$,  $a, b\in\Z$. More recently, Krause--Mirek--Tao \cite{KMT} established pointwise convergence  for the polynomial Furstenberg--Weiss averages \cite{Fur1,FurWei}
 corresponding to $P_{1,1}(n)=n$ and $P_{1, 2}(n)=P(n)$, ${\rm deg }\,P\ge2$.
 
 \subsubsection{Norm convergence} Except for these few cases, there are no other results concerning pointwise convergence for the averages \eqref{eq:41}. The situation is completely different, however, for the question of norm convergence, which is much better understood. 
 
 A breakthrough paper of Walsh \cite{W} (see also \cite{A1}) gives a complete picture of $L^2(X)$ norm convergence of the averages \eqref{eq:41} for any $T_1,\ldots, T_d\in \mathbb G$ where $\mathbb G$ is a nilpotent group
of transformations of a probability space.  Prior to this, there was an extensive body of research towards
establishing $L^2(X)$ norm convergence, including groundbreaking works of Host--Kra
\cite{HK}, Ziegler \cite{Z1}, Bergelson \cite{Ber0},
and Leibman \cite{Leibman}. See also \cite{A2, CFH, FraKra, HK1, Tao} and the survey articles \cite{Ber1,Ber2,Fra} for more details and references, including a comprehensive historical background.

\subsubsection{Additional remarks} Bergelson and Leibman \cite{BeLe}  showed that convergence may fail if the transformations $T_1,\ldots, T_d$ generate a solvable group, so the
nilpotent setting is probably the appropriate setting for Conjecture \ref{con:1}. The restriction $p>1$ is necessary in the case of nonlinear polynomials as was shown in \cite{BM, LaV1}. 

If $(X, \mathcal B(X), \mu)$ is a probability space and the family of measure preserving transformations $(T_1,\ldots,T_{d_1})$ is {\it{totally ergodic}}, then
Theorem \ref{thm:main}(ii) implies that
\begin{align}
\label{eq:45}
\lim_{N\to\infty}A_{N; X}^{P_{1}, \ldots, P_{d_1}}(f)(x)=\int_Xf(y)d\mu(y)
\end{align}
$\mu$-almost everywhere on $X$. We recall that a family of measure preserving transformations $(T_1,\ldots,T_{d_1})$ is called ergodic on $X$ if $T_j^{-1}(B)=B$ for all $j\in\{1,\ldots, d_1\}$ implies $\mu(B)=0$ or $\mu (B)=1$ and is called totally ergodic if the family $(T_1^n,\ldots,T_{d_1}^n)$ is ergodic for all $n\in\Z_+$.

\subsection{The universal step-two group $\G_0$}\label{setup} The proof of Theorem \ref{thm:main} will follow from our second main result, Theorem \ref{thm:main1} below,  for averages on  universal nilpotent groups of step two. We start with some definitions. For integers $d\geq 1$, we define
\begin{equation*}
Y_d:=\{(l_1,l_2)\in\mathbb{Z}\times\mathbb{Z}:0\leq l_2<l_1\leq d\}
\end{equation*}
and the ``universal'' step-two nilpotent Lie groups $\G_0^\#=\G_0^\#(d)$
\begin{equation}\label{gro1}
\G_0^\#:=\{(x_{l_1l_2})_{(l_1,l_2)\in Y_d}:x_{l_1l_2}\in\mathbb{R}\},
\end{equation}
with the group multiplication law
\begin{equation}\label{gro2}
[x\cdot y]_{l_1l_2}:=
\begin{cases}
x_{l_10}+y_{l_10}&\text{ if }l_1\in \{1,\ldots, d\}\text{ and }l_2=0,\\
x_{l_1l_2}+y_{l_1l_2}+x_{l_10}y_{l_20}&\text{ if }l_1\in\{1,\ldots, d\}\text{ and }l_2\in\{1, \ldots, l_1-1\}.
\end{cases}
\end{equation}

Alternatively, we can also define the group $\G_0^\#$ as the set of elements 
\begin{equation}\label{picu4}
g=(g^{(1)},g^{(2)}),\qquad g^{(1)}=(g_{l_10})_{l_1\in\{1,\ldots, d\}}\in\R^d,\qquad g^{(2)}=(g_{l_1l_2})_{(l_1,l_2)\in Y'_d}\in\R^{d'},
\end{equation}
where $d':=d(d-1)/2$ and $Y'_{d}:=\{(l_1,l_2)\in Y_d:\,l_2\geq 1\}$. Letting 
\begin{equation}\label{picu4.1}
R_0:\R^d\times\R^d\to\R^{d'}\quad\text{ denote the bilinear form }\quad [R_0(x,y)]_{l_1l_2}:=x_{l_10}y_{l_20},
\end{equation}
we notice that the product rule in the group $\G_0^\#$ is given by
\begin{equation}\label{picu4.2}
[g\cdot h]^{(1)}:=g^{(1)}+h^{(1)},\qquad [g\cdot h]^{(2)}:=g^{(2)}+h^{(2)}+R_0(g^{(1)},h^{(1)})
\end{equation}
if $g=(g^{(1)},g^{(2)})$ and $h=(h^{(1)},h^{(2)})$. For any $g=(g^{(1)},g^{(2)})\in \G_0^{\#}$, its inverse is given by
\begin{equation*}
g^{-1}=\big(-g^{(1)}, -g^{(2)}+R_0(g^{(1)}, g^{(1)})\big).
\end{equation*}
The second variable of $g=(g^{(1)},g^{(2)})\in \G_0^{\#}$ is called the central variable. Based on the product structure \eqref{picu4.2} of the group $\G_0^\#$, it is not difficult to see that $g\cdot h=h\cdot g$ for any $g=(g^{(1)},g^{(2)})\in \G_0^{\#}$ and $h=(0, h^{(2)})\in \G_0^{\#}$. 

Let $\G_0=\G_0(d)$ denote the discrete subgroup
\begin{align}
\label{eq:48}
\G_0:=\G_0^\#\cap\Z^{|Y_d|}.
\end{align}
 Let $A_0:\mathbb{R}\to\G_0^\#$ denote the canonical  polynomial map (or the moment curve on $\G_0^\#$)
\begin{equation}\label{tra3}
[A_0(x)]_{l_1l_2}:=
\begin{cases}
x^{l_1}&\text{ if }l_2=0,\\
0&\text{ if }l_2\neq 0,
\end{cases}
\end{equation}
and notice that $A_0(\Z)\subseteq\G_0$. For $x=(x_{l_1l_2})_{(l_1,l_2)\in Y_d}\in \G_0^\#$ and $\Lambda\in(0,\infty)$, we define 
\begin{equation}\label{tra3.5}
\Lambda\circ x:=(\Lambda^{l_1+l_2}x_{l_1l_2})_{(l_1,l_2)\in Y_d}\in \G_0^\#.
\end{equation}
Notice that the dilations $\Lambda\circ$ are group homomorphisms on the group $\G_0$ that are compatible with the map $A_0$, i.e. $\Lambda\circ A_0(x)=A_0(\Lambda x)$. 

Let $\chi:\mathbb{R}\to[0,1]$ be a smooth function  supported on the interval $[-2, 2]$. Given any real number $N\ge1$ and a function $f:\G_0\to\C$, we can define a smoothed average along the moment curve $A_0$ by the formula
\begin{align}
\label{eq:47}
M_{N}^{\chi}(f)(x):=\sum_{n\in\Z}N^{-1}\chi(N^{-1}n)f(A_0(n)^{-1}\cdot x), \qquad x\in \G_0.
\end{align}

The main advantage of working on the group $\G_0$ with the polynomial map $A_0$ is the presence of the compatible dilations $\Lambda\circ$ defined in \eqref{tra3.5}, which lead to a natural family of associated balls. This can be efficiently exploited by noting that $M_{N}^{\chi}$ is a convolution operator on $\G_0$.

The convolution of functions on the group $\G_0$ is defined by the formula
\begin{equation}\label{convoDef}
(f\ast g) (x):=\sum_{y\in\G_0}f(y^{-1}\cdot x)g(y)=\sum_{z\in\G_0}f(z)g(x\cdot z^{-1}).
\end{equation}
Then it is not difficult to  see that $M_{N}^{\chi}(f)(x)=f*G_N^{\chi}(x)$, where
\begin{align}
\label{eq:46}
G_N^{\chi}(x):= \sum_{n\in\Z}N^{-1}\chi(N^{-1}n)\ind{\{A_0(n)\}}(x),
\qquad x\in \G_0.
\end{align}

We are now ready to state our second main result.

\begin{theorem}[Boundedness on $\G_0$]
\label{thm:main1}
Let $\G_0=\G_0(d)$, $d\geq 1$, be the discrete nilpotent group defined in \eqref{eq:48} and $A_0$ the polynomial sequence defined in \eqref{tra3}. Then 

(i) (Maximal estimates) If $1<p\le \infty$ and $f\in\ell^p(\G_0)$ then
\begin{align}
\label{eq:39}
\big\|\sup_{N\geq 1}|M_{N}^{\chi}(f)|\big\|_{\ell^p(\G_0)}\lesssim_p\|f\|_{\ell^p(\G_0)},
\end{align}
where $M_N^\chi$ is defined as in \eqref{eq:47}.

(ii) (Long variational estimates)
If $1<p< \infty$ and $\rho>\max\big\{p, \frac{p}{p-1}\big\}$, and $\tau\in(1,2]$ then
\begin{align}
\label{eq:49}
\big\|V^{\rho}\big(M_{N}^{\chi}(f):N\in\mathbb D_{\tau}\big)\big\|_{\ell^p(\G_0)}\lesssim_{p, \rho,\tau}\|f\|_{\ell^p(\G_0)},
\end{align}
where $\mathbb D_{\tau}:=\{\tau^n:n\in\N\}$. See \eqref{eq:44} for the definition of the $\rho$-variation seminorms $V^{\rho}$.

(iii) (Singular integrals) If $K:\mathbb{R}\to\mathbb{R}$ is a Calder{\'o}n--Zygmund kernel as in \eqref{CalZyg}, $f:\mathbb{G}_0\to\mathbb{C}$ is a (compactly supported) function, and
\begin{equation*}
H_0f(g):=\sum_{n\in\mathbb{Z}} K(n)f(A_0^{-1}(n)\cdot g),\qquad g\in\mathbb{G}_0,
\end{equation*}
then
\begin{equation*}
\|H_0f\|_{\ell^2(\mathbb{G})}\lesssim \|f\|_{\ell^2(\mathbb{G})}.
\end{equation*}

\end{theorem}

\subsection{Remarks and overview of the proof} We discuss now some of the main ideas in the proofs of Theorems \ref{thm:main} and \ref{thm:main1}.

\subsubsection{The Calder{\'o}n transference principle} One can show that Theorem \ref{thm:main} is a consequence of Theorem \ref{thm:main1}
upon performing lifting arguments and adapting the Calder{\'o}n transference principle. Indeed, if $\G^\#$ is a connected and simply connected nilpotent Lie group of step 2, with Lie algebra $\mathcal{G}$, then one can choose so-called exponential coordinates of the second kind associated to a Malcev basis of the Lie algebra $\mathcal{G}$ (see \cite{CoGr}, Sec. 1.2) in such a way that 
\begin{equation*}
\G^\#\simeq\{(x,y)\in\mathbb{R}^{b_1}\times\mathbb{R}^{b_2}:(x,y)\cdot (x',y')=(x+x',y+y'+R(x,x')\},
\end{equation*}
where $b_1,b_2\in\Z_+$ depend on the Lie algebra $\mathcal{G}$ and $R:\mathbb{R}^{b_1}\times\mathbb{R}^{b_1}\to\mathbb{R}^{b_2}$ is a bilinear form.

Moreover, if $\G\leq\G^\#$ is a discrete co-compact subgroup, then one can choose the Malcev basis such that the discrete subgroup $\G$ is identified with the integer lattice $\Z^b=\Z^{b_1}\times\Z^{b_2}$ (see \cite{CoGr}, Thm. 5.1.6 and Prop. 5.3.2). Recall that $A:\mathbb{Z}\to\G$ is a polynomial sequence satisfying $A(0)=1$. The main point is that one can choose $d$ sufficiently large and a group morphism $T:\G_0\to\G^\#$ such that
\begin{equation*}
A(n)=T(A_0(n))\qquad\text{ for any }n\in\mathbb{Z}.
\end{equation*}
Then one can use this group morphism to transfer bounds on operators on the universal group $\G_0$ to bounds on operators on the group $\G$. Theorem \ref{thm:main} is thus a consequence of Theorem \ref{thm:main1} and our main goal therefore is to prove Theorem \ref{thm:main1}.

\subsubsection{The variation spaces $V^\rho$} For any family $(a_t: t\in\mathbb I)$ of elements of $\C$
indexed by a totally ordered set $\mathbb I$, and any exponent
$1 \leq \rho < \infty$, the $\rho$-variation seminorm is defined by
\begin{align}
\label{eq:44}
V^{\rho}( a_t: t\in\mathbb I):=
\sup_{J\in\Z_+} \sup_{\substack{t_{0}<\dotsb<t_{J}\\ t_{j}\in\mathbb I}}
\Big(\sum_{j=0}^{J-1}  |a(t_{j+1})-a(t_{j})|^{\rho} \Big)^{1/\rho},
\end{align}
where the supremum is taken over all finite increasing sequences in
$\mathbb I$. It is easy to see that $\rho\mapsto V^{\rho}$ is non-increasing, and
for every $t_0\in \mathbb I$ one has
\begin{align}
\label{eq:36}
\sup_{t\in \mathbb I}|a_t|\le |a_{t_0}|+  V^\rho( a_t: t\in\mathbb I)\le \sup_{t\in \mathbb I}|a_t| +  V^\rho( a_t: t\in\mathbb I). 
\end{align}

In particular, the maximal estimate \eqref{eq:39} follows from the variational estimate \eqref{eq:49}. The main point of proving stronger variational estimates such as \eqref{eq:49}, with general parameters $\tau\in(1,2]$, is that it gives an elegant path to deriving pointwise ergodic theorems (which would not follow directly just from maximal estimates such as \eqref{eq:39}). At the same time, the analysis of variational inequalities has many similarities with the analysis of maximal inequalities, and is not substantially more difficult. This is due in large part to the Rademacher--Menshov inequality (see \cite[Lemma 2.5]{MSZ2}): for any 
$2\le \rho<\infty$ and $j_0, m\in\N$ so that $j_0< 2^m$  and any sequence of complex numbers $(\mathfrak a_k: k\in\N)$  we have 
\begin{equation}\label{maj1}
V^{\rho}( \mathfrak a_j: j_0\leq j\leq 2^m)\leq \sqrt{2}\sum_{i=0}^m\bigg(\sum_{j\in[j_02^{-i},2^{m-i}-1]\cap\Z}\Big|\mathfrak a_{(j+1)2^i}-\mathfrak a_{j2^i})\Big|^2\bigg)^{1/2}.
\end{equation}

\subsubsection{$\ell^p$ theory} The problem of passing from $\ell^2$ estimates to $\ell^p$ estimates in the context of discrete polynomial averages has been investigated extensively in recent years (see, for example, \cite{MST2} and the references therein).

The full $\ell^p(\G_0)$ bounds in Theorem \ref{thm:main1} rely on first proving $\ell^2(\G_0)$ bounds. In fact, we first establish \eqref{eq:49} for $p=2$ and $\rho>2$. Then we use the positivity of the operators $M_N^\chi$ (i.e. $M_{N}^{\chi}(f)\ge0$ if $f\ge0$) to prove the maximal operator bounds \eqref{eq:39} for all $p\in(1,\infty]$. Finally, we use vector-valued interpolation between the bounds \eqref{eq:49} with $p=2$ and $\rho>2$ and \eqref{eq:39} with $p\in(1,\infty]$ to complete the proof of Theorem \ref{thm:main1}.

\subsubsection{Some technical remarks} Theorem \ref{thm:main1} (i) and (ii)  extends the results of \cite{MST2, MSZ3} to the non-co\-mmu\-tative, nilpotent setting. Its conclusions remain true for rough averages, i.e. when $\chi=\ind{[-1, 1]}$ in \eqref{eq:47}, but it is more convenient to work with smooth averages. 

The restriction $p>1$ in Theorem \ref{thm:main1} (i) and (ii) is sharp due to \cite{BM, LaV1}. However, the range of $\rho>\max\big\{p, \frac{p}{p-1}\big\}$ is only sharp when $p=2$ due to L{\'e}pingle's inequality \cite{Le}. One could hope to improve this to the full range $\rho>2$, but we do not address this here since the limited range $\rho>\max\big\{p, \frac{p}{p-1}\big\}$ is already sufficient for us to establish Theorem \ref{thm:main}.  

The restriction $p=2$ in the singular integral bounds in part (ii) is probably not necessary. In the commutative case one can prove boundedness in the full range $p\in(1,\infty)$ (see \cite{IW}), but the proof depends on exploiting certain Fourier multipliers and we do not know at this time if a similar definitive result holds in the nilpotent case.

\subsection{The main difficulty and a nilpotent circle method} Bourgain's seminal papers \cite{Bo2, Bo3, Bo1} generated a large amount of research and progress in the field. Many other discrete operators have been analyzed by many authors motivated by problems in Analysis and Ergodic Theory. See, for example, \cite{BM, IMSW, IW, Kr, KMT, LaV1, MSW0, MST2, MSZ2, MSZ3, Pi1, Pi2, SW0} for some results of this type and more references. A common feature of all of these results, which plays a crucial role in the proofs, is that one can use Fourier analysis techniques, in particular, the powerful framework of the classical circle method, to perform the analysis.

Our situation in Theorem \ref{thm:main1} is different. The main conceptual issue is that 
there is no good Fourier transform on nilpotent groups, compatible with the structure of the underlying convolution operators and at the level of analytical precision of the classical circle method. At a more technical level, there is no good resolution of the delta function compatible with the group multiplication on the group $\G_0$. This prevents us from using a naive implementation of the circle method. The classical delta function resolution
\begin{align*}
\ind{\{0\}}(x^{-1}\cdot y)=\int_{\T^d\times\T^{d'}}\ex((y^{(1)}-x^{(1)}){.}\theta^{(1)})\ex((y^{(2)}-x^{(2)}){.}\theta^{(2)})\,d\theta^{(1)}d\theta^{(2)},
\end{align*}
does not detect the group  multiplication correctly.
Here $(y^{(1)}-x^{(1)}){.}\theta^{(1)}$ and $(y^{(2)}-x^{(2)}){.}\theta^{(2)}$ denote the usual scalar product of vectors in $\R^{d}$ and $\R^{d'}$, respectively.

These issues lead to very significant difficulties in the proof and require substantial new ideas.
Our main new construction in \cite{IMMS} is what we call a \textit{nilpotent circle method}, an iterative procedure, starting from the center of the group and moving down along its central series. At every stage we identify ``minor arcs", and bound their contributions using Weyl's inequalities (the classical Weyl inequality as well as a nilpotent Weyl inequality which was proved in \cite{IMW}).  The final stage involves ``major arcs" analysis, which relies on a combination of continuous harmonic analysis on groups $\G_0^{\#}$ and arithmetic harmonic analysis over finite integer rings modulo $Q\in\Z_+$. We outline this procedure in Section \ref{outline} below.

At the implementation level, classical Fourier techniques are replaced with almost or\-tho\-gonality methods based on exploiting high order $T^\ast T$ arguments for
operators defined on  the discrete group $\G_0$. Investigating high powers of $T^\ast T$ (i.e. $(T^\ast T)^r$ for a large $r\in\Z_+$)  is consistent with a general heuristic lying behind the proof of Waring-type problems, which says that the more variables that occur in  Waring-type equations, the easier it is to find solutions, and we are able to make this heuristic rigorous in our problem. Manipulating the parameter $r$, by taking $r$ to be very large, we can always decide how many variables we have at our disposal, making our operators ``smoother and smoother''.

\subsection{General discrete nilpotent groups} The primary goal is, of course, to remove the restriction that the discrete nilpotent groups $\G$ in Theorem \ref{thm:main} are of step 2, and thus establish the full Conjecture \ref{con:1} in the linear $m=1$ case for arbitrary invertible measure-preserving transformations $T_1,\ldots, T_d$ that generate a nilpotent group of any step $k\ge 2$. The iterative argument we outline in Section \ref{outline} below could, in principle, be extended to higher step groups, at least as long as the group and the polynomial sequence have suitable ``universal"-type structure, as one could try to go down along the central series of the group and prove minor arcs and transition estimates at every stage.

However, this is only possible if one can prove suitable analogues of the nilpotent Weyl's inequalities in Proposition \ref{minarcs} on general nilpotent groups of step $k\ge3$. The point is to have a small (not necessarily optimal, but nontrivial) gain for bounds on oscillatory sums over many variables, corresponding to the kernels of high power $(T^\ast T)^r$ operators, whenever frequencies are restricted to the minor arcs. In our case, the formulas are explicit, see the identities \eqref{pro0.3.5}, and we can use ideas of Davenport \cite{Dav} and Birch \cite{Bi}  for Diophantine forms in many variables to control the induced oscillatory sums, but the analysis seems to be more complicated for the higher step nilpotent groups.  

This is an interesting problem in its own right, corresponding to Waring-type problems on nilpotent groups.
A qualitative variant of the Waring problem on nilpotent groups was recently investigated in \cite{HU1, HU2}, see also the references given there. We prove a quantitative version on our nilpotent group $\G_0$ in Theorem \ref{thm:1} below.

\subsection{Organization} The rest of this paper is organized as follows: in section 2 we present several nilpotent Weyl estimates proved in \cite{IMW}, which play a key role in the analysis of minor arcs. In section 3 we outline our main new method, the nilpotent circle method, developed in \cite{IMMS} to prove maximal and variational estimates on nilpotent groups. In section 4 we prove a new Waring-type theorem on the nilpotent group $\G_0$, as an application of the nilpotent Weyl estimates discussed earlier.

\section{A nilpotent Weyl inequality on the group $\G_0$}\label{sus1}

In this section we derive explicit formulas used in high order $T^\ast T$ arguments and discuss a key ingredient in our analysis, namely Weyl inequalities on the group $\G_0$.

\subsection{High order $T^\ast T$ arguments and product kernels} Many of our $\ell^2(\G_0)$ estimates will be  based on high order $T^\ast T$ arguments. Assume that 
\begin{equation*}
S_1,T_1,\ldots,S_r,T_r:\ell^2(\G_0)\to\ell^2(G_0)
\end{equation*}
are convolution operators defined by some $\ell^1(\G_0)$ kernels $L_1,K_1,\ldots, L_r,K_r:\G_0\to\C$, i.e. $S_j f = f \ast L_j$ and $T_j f = f \ast K_j$ for $j \in \{1,\ldots,r\}$. Then the adjoint operators $S_1^\ast,\ldots, S_r^\ast$ are also convolution operators, defined by the kernels $L_1^\ast,\ldots, L_r^\ast$ given by 
\begin{equation*}
L_j^\ast(g):=\overline{L_j(g^{-1})}.
\end{equation*}
Moreover, using \eqref{convoDef}, for any $f\in\ell^2(\G_0)$ and $x\in\G_0$, we have
\begin{equation}\label{pro11.1}
\begin{split}
(S_1^\ast T_1\ldots S_r^\ast T_rf)(x)
=\sum_{h_1,g_1,\ldots,h_r,g_r\in\G_0}\Big\{\prod_{j=1}^rL_j^\ast(h_j)K_j(g_j)\Big\}f(g_r^{-1}\cdot h_r^{-1}\cdot\ldots\cdot g_1^{-1}\cdot h_1^{-1}\cdot x).
\end{split}
\end{equation}
In other words $(S_1^\ast T_1\ldots S_r^\ast T_rf)(x)=(f\ast A^r)(x)$,
where the kernel $A^r$ is given by
\begin{equation}\label{pro11}
A^r(y):=\sum_{h_1,g_1,\ldots,h_r,g_r\in\G_0}\Big\{\prod_{j=1}^r\overline{L_j(h_j)}K_j(g_j)\Big\}\ind{\{0\}}(g_r^{-1}\cdot h_r\cdot\ldots\cdot g_1^{-1}\cdot h_1\cdot y).
\end{equation}

To use these formulas we decompose $h_j=(h_j^{(1)},h_j^{(2)}),\,g_j=(g_j^{(1)},g_j^{(2)})$ as in \eqref{picu4}. Then
\begin{equation}\label{pro15}
[h_1^{-1}\cdot g_1\cdot\ldots\cdot h_r^{-1}\cdot g_r]^{(1)}=\sum_{1\leq j\leq r}(-h_j^{(1)}+g_j^{(1)}),
\end{equation}
\begin{equation}\label{pro15.5}
\begin{split}
[h_1^{-1}\cdot g_1\cdot\ldots\cdot h_r^{-1}\cdot g_r]^{(2)}&=\sum_{1\leq j\leq r}\big\{-(h_j^{(2)}-g_j^{(2)})+R_0(h_j^{(1)},h_j^{(1)}-g_j^{(1)})\big\}\\
&+\sum_{1\leq l<j\leq r}R_0(-h_l^{(1)}+g_l^{(1)},-h_j^{(1)}+g_j^{(1)}),
\end{split}
\end{equation}
as a consequence of applying \eqref{picu4.2} inductively. 

In many of our applications the operators $S_1,T_1,\ldots,S_r,T_r$ are equal and, more importantly, are defined by a kernel $K$ that has product structure, i.e.
\begin{equation}\label{pro15.6}
\begin{split}
&S_1f=T_1f=\ldots=S_rf=T_rf=f\ast K,\\
&K(g)=K(g^{(1)},g^{(2)})=K^{(1)}(g^{(1)})K^{(2)}(g^{(2)}).
\end{split}
\end{equation}
In this case we can derive an additional formula for the kernel $A^r$. We use the identity
\begin{equation*}
\ind{\{0\}}(x^{-1}\cdot y)=\int_{\T^d\times\T^{d'}}\ex((y^{(1)}-x^{(1)}){.}\theta^{(1)})\ex((y^{(2)}-x^{(2)}){.}\theta^{(2)})\,d\theta^{(1)}d\theta^{(2)},
\end{equation*}
where $\ex(z):=e^{2\pi i z}$. The formula \eqref{pro11} shows that
\begin{equation}\label{pro15.7}
A^r(y)=\int_{\T^d\times\T^{d'}}\ex\big(y^{(1)}.\theta^{(1)}\big)\ex\big(y^{(2)}.\theta^{(2)}\big)\Sigma^r\big(\theta^{(1)},\theta^{(2)}\big)\,d\theta^{(1)}d\theta^{(2)},
\end{equation}
where 
\begin{equation*}
\begin{split}
\Sigma^r\big(\theta^{(1)},\theta^{(2)}\big)&:=\sum_{h_j,g_j\in\G_0}\Big\{\prod_{j=1}^r\overline{K(h_j)}K(g_j)\Big\}
\prod_{i=1}^2\ex\big(-[h_1^{-1}\cdot g_1\cdot\ldots\cdot h_r^{-1}\cdot g_r]^{(i)}{.}\theta^{(i)}\big).
\end{split}
\end{equation*}

Recalling the product formula \eqref{pro15.6} we can write
\begin{equation}\label{pro15.9}
\Sigma^r\big(\theta^{(1)},\theta^{(2)}\big)=\Pi^r\big(\theta^{(1)},\theta^{(2)}\big)\Omega^r\big(\theta^{(2)}\big),
\end{equation}
for any $(\theta^{(1)},\theta^{(2)})\in\T^d\times\T^{d'}$,
where
\begin{equation}\label{pro15.10}
\begin{split}
\Pi^r&\big(\theta^{(1)},\theta^{(2)}\big):=\sum_{h_j^{(1)},g_j^{(1)}\in\Z^d}\Big\{\prod_{j=1}^r\overline{K^{(1)}(h_j^{(1)})}K^{(1)}(g_j^{(1)})\Big\}\ex\big(\theta^{(1)}{.}\sum_{1\leq j\leq r}(h_j^{(1)}-g_j^{(1)})\big)\\
&\times\ex\Big(-\theta^{(2)}{.}\big\{\sum_{1\leq j\leq r}R_0(h_j^{(1)},h_j^{(1)}-g_j^{(1)})+\sum_{1\leq l<j\leq r}R_0(-h_l^{(1)}+g_l^{(1)},-h_j^{(1)}+g_j^{(1)})\big\}\Big)
\end{split}
\end{equation}
and
\begin{equation}\label{pro15.11}
\begin{split}
\Omega^r\big(\theta^{(2)}\big)&:=\sum_{h_j^{(2)},g_j^{(2)}\in\Z^{d'}}\Big\{\prod_{j=1}^r\overline{K^{(2)}(h_j^{(2)})}K^{(2)}(g_j^{(2)})\Big\}\ex\big(\theta^{(2)}{.}\sum_{1\leq j\leq r}(h_j^{(2)}-g_j^{(2)})\big)\\
&=\Big|\sum_{g^{(2)}\in\Z^{d'}}K^{(2)}(g^{(2)})\ex\big(-\theta^{(2)}{.}g^{(2)}\big)\Big|^{2r}.
\end{split}
\end{equation}

\subsection{Weyl estimates} After applying high order $T^\ast T$ arguments we often need to estimate exponential sums and oscillatory integrals involving polynomial phases. With the notation in Section \ref{setup}, for $r\geq 1$ let $D,\widetilde{D}:\mathbb{R}^r\times\mathbb{R}^r\to\G_0^\#$ be defined by
\begin{equation}\label{pro0.3.5}
\begin{split}
&D((n_1,\ldots,n_r),(m_1,\ldots,m_r)):=A_0(n_1)^{-1}\cdot A_0(m_1)\cdot\ldots\cdot A_0(n_r)^{-1}\cdot A_0(m_r),\\
&\widetilde{D}((n_1,\ldots,n_r),(m_1,\ldots,m_r)):=A_0(n_1)\cdot A_0(m_1)^{-1}\cdot\ldots\cdot A_0(n_r)\cdot A_0(m_r)^{-1}.
\end{split}
\end{equation}
By definition, we have
\begin{equation*}
[A_0(n)]_{l_1l_2}=\begin{cases}
n^{l_1}&\text{ if }l_2=0,\\
0&\text { if }l_2\geq 1,
\end{cases}
\qquad
[A_0(n)^{-1}]_{l_1l_2}=\begin{cases}
-n^{l_1}&\text{ if }l_2=0,\\
n^{l_1+l_2}&\text { if }l_2\geq 1.
\end{cases}
\end{equation*}
Thus, using \eqref{pro15} and \eqref{pro15.5}, for $x=(x_1,\ldots,x_r)\in\mathbb{R}^r$ and $y=(y_1,\ldots,y_r)\in\mathbb{R}^r$ one has
\begin{equation}\label{pro0.4}
[D(x,y)]_{l_1l_2}=\begin{cases}
\sum\limits_{j=1}^r(y_j^{l_1}-x_j^{l_1})&\text{ if }l_2=0,\\
\sum\limits_{1\leq j_1<j_2\leq r}(y_{j_1}^{l_1}-x_{j_1}^{l_1})(y_{j_2}^{l_2}-x_{j_2}^{l_2})+\sum\limits_{j=1}^r(x_j^{l_1+l_2}-x_j^{l_1}y_j^{l_2})&\text{ if }l_2\geq 1,
\end{cases}
\end{equation}
and
\begin{equation}\label{pro0.5}
[\widetilde{D}(x,y)]_{l_1l_2}=\begin{cases}
\sum\limits_{j=1}^r(x_j^{l_1}-y_j^{l_1})&\text{ if }l_2=0,\\
\sum\limits_{1\leq j_1<j_2\leq r}(x_{j_1}^{l_1}-y_{j_1}^{l_1})(x_{j_2}^{l_2}-y_{j_2}^{l_2})+\sum\limits_{j=1}^r(y_j^{l_1+l_2}-x_j^{l_1}y_j^{l_2})&\text{ if }l_2\geq 1.
\end{cases}
\end{equation}

For $P\in\Z_+$ assume $\phi_P^{(j)},\psi_P^{(j)}:\mathbb{R}\to\mathbb{R}$, $j\in\{1,\ldots,r\}$, are $C^1(\R)$ functions with the properties
\begin{equation}\label{pro0.1}
\sup_{1\le j\le r}\big[\big|\phi_P^{(j)}\big|+\big|\psi_P^{(j)}\big|\big]\leq \ind{[-P,P]},\qquad
\sup_{1\le j\le r}\int_{\mathbb{R}}\big|[\phi^{(j)}_P]'(x)\big|+\big|[\psi^{(j)}_P]'(x)\big|\,dx\leq 1.
\end{equation}
For $\theta=(\theta_{l_1l_2})_{(l_1,l_2)\in Y_d}\in\mathbb{R}^{|Y_d|}$, $r\in\Z_+$, and $P\in \Z_+$ let
\begin{equation*}
S_{P,r}(\theta)=\sum_{n,m\in\Z^r}\ex(- D(n,m){.}\theta)\Big\{\prod_{j=1}^r\phi_P^{(j)}(n_j)\psi_P^{(j)}(m_j)\Big\}
\end{equation*}
and
\begin{equation*}
\widetilde{S}_{P,r}(\theta)=\sum_{n,m\in\Z^r}\ex(- \widetilde{D}(n,m){.}\theta)\Big\{\prod_{j=1}^r\phi_P^{(j)}(n_j)\psi_P^{(j)}(m_j)\Big\},
\end{equation*}
where $D$ and $\widetilde{D}$ are defined as in \eqref{pro0.4}--\eqref{pro0.5}. 

The following key estimates are proved in \cite[Proposition 5.1 and Lemma 3.1]{IMW}:

\begin{proposition}\label{minarcs} (i) (Nilpotent Weyl estimate) For any $\varepsilon>0$ there is $r=r(\varepsilon,d)\in\Z_+$ sufficiently large such that for all $P\in \Z_+$ we have
\begin{equation}\label{pro0.2}
|S_{P,r}(\theta)|+|\widetilde{S}_{P,r}(\theta)|\lesssim_\varepsilon P^{2r}P^{-1/\varepsilon},
\end{equation}
provided that there is $(l_1,l_2)\in Y_d$ and an irreducible fraction $a/q\in\mathbb{Q}$, $q\in\Z_+$, such that
\begin{equation}\label{pro0.3}
|\theta_{l_1l_2}-a/q|\leq 1/q^2\text{ and }q\in[P^{\varepsilon},P^{l_1+l_2-\varepsilon}].
\end{equation}

(ii) (Nilpotent Gauss sums) For any irreducible fraction $a/q\in\mathbb Q^{|Y_d|}$, $a=(a_{l_1l_2})_{(l_1,l_2)\in Y_d}\in\mathbb{Z}^{|Y_d|}$, $q\in\Z_+$, we define the arithmetic coefficients
\begin{equation}\label{pro0.6}
G(a/q):=q^{-2r}\sum_{v,w\in \Z_q^r} \ex\big(- D(v,w){.}(a/q) \big),\qquad \widetilde{G}(a/q):=q^{-2r}\sum_{v,w\in \Z_q^r}\ex\big( -\widetilde{D}(v,w){.}(a/q) \big).
\end{equation}
Then for any $\varepsilon>0$ there is $r=r(\varepsilon,d)\in\Z_+$ sufficiently large such that
\begin{equation}\label{pro0.7}
|G(a/q)|+|\widetilde{G}(a/q)|\lesssim_\varep q^{-1/\varepsilon}.
\end{equation}
\end{proposition}

We also need a related integral estimate, see Lemma 5.4 in \cite{IMW}:

\begin{proposition}\label{minarcscon} Given $\varepsilon>0$ there is $r=r(\varepsilon,d)$ sufficiently large as in Proposition \ref{minarcs} such that
\begin{equation}\label{pro4.2}
\begin{split}
\Big|\int_{\R^r\times\R^r}\Big\{\prod_{j=1}^r\phi_j(x_j)\psi_j(y_j)\Big\}\ex(- D(x,y){.}\beta \big)\,dxdy\Big|\lesssim \langle\beta\rangle^{-1/\varepsilon},\\
\Big|\int_{\R^r\times\R^r}\Big\{\prod_{j=1}^r\phi_j(x_j)\psi_j(y_j)\Big\}\ex(- \widetilde{D}(x,y){.}\beta \big)\,dxdy\Big|\lesssim \langle\beta\rangle^{-1/\varepsilon},
\end{split}
\end{equation}
for any $\beta \in \mathbb{R}^{|Y_d|}$; here and later on we use the Japanese bracket notation $\langle \beta \rangle:=(1+|\beta|^2)^{1/2}$, and for any $C^1(\R)$ functions $\phi_1,\psi_1,\ldots,\phi_r,\psi_r:\R\to\mathbb{C}$ satisfying, for any $j\in\{1,\ldots,r\}$, the bounds
\begin{equation*}
|\phi_j|+|\psi_j|\leq\ind{[-1,1]}(x),\qquad \int_\R\big[|\partial_x\phi_j(x)|+|\partial_x\psi_j(x)|\big]\,dx\leq 1.
\end{equation*}
\end{proposition}

These statements should be compared with classical Weyl-type estimates, which are proved for example in \cite[Proposition 1]{SW0}:

\begin{proposition}\label{minarcscom} (i) Assume that $P\geq 1$ is an integer and $\phi_P:\mathbb{R}\to\mathbb{R}$ is a $C^1(\R)$ function satisfying
\begin{equation}\label{comm1}
|\phi_P|\leq \ind{[-P,P]},\qquad \int_{\mathbb{R}}\big|\phi'_P(x)\big|\,dx\leq 1.
\end{equation}
Assume that $\varepsilon>0$ and $\theta=(\theta_1,\ldots,\theta_d)\in\R^d$ has the property that there is $l\in\{1,\ldots,d\}$ and an irreducible fraction $a/q\in\mathbb{Q}$ with $q\in\Z_+$, such that
\begin{equation}\label{comm3}
|\theta_{l}-a/q|\leq 1/q^2\,\,\text{ and }\,\,q\in[P^{\varepsilon},P^{l-\varepsilon}].
\end{equation}
Then there is a constant $\overline{C}=\overline{C}_d\geq 1$ such that 
\begin{equation}\label{comm4}
\Big|\sum_{n\in\mathbb{Z}}\phi_P(n)
\ex\big( - (\theta_1n+\ldots+\theta_dn^d) \big)
\Big|\lesssim_\varep P^{1-\varepsilon/\overline{C}}.
\end{equation}

(ii) For any irreducible fraction $\theta=a/q\in(\Z/q)^d$, $a=(a_1,\ldots,a_d)\in\Z^d$, $q\in\Z_+$, we have
\begin{equation}\label{comm4.5}
\Big|q^{-1}\sum_{n\in \Z_q}
\ex\big( - (\theta_1n+\ldots+\theta_dn^d) \big)
\Big|\lesssim q^{-1/\overline{C}}.
\end{equation}
\end{proposition}

Notice a formal similarity between Proposition \ref{minarcs} and \ref{minarcscom}. They both involve a small but non-trivial gain of a power of $P$ as soon as one of the coefficients of the relevant polynomials is far from rational numbers with small denominators. These estimates can therefore be used efficiently to estimate minor arcs contributions.

We note, however, that the proof of the nilpotent Weyl estimates in Proposition \ref{minarcs} is much more involved than the proof of Proposition \ref{minarcscom}. It relies on some classical ideas of Davenport \cite{Dav} and Birch \cite{Bi} on treating polynomials in many variables, but one has to identify and exploit  suitable non-degeneracy properties of the explicit (but complicated) polynomials $D$ and $\widetilde{D}$ in \eqref{pro0.4}--\eqref{pro0.5} to make the proof work. All the details of the proof are provided in \cite[Section 5]{IMW}.

\section{A nilpotent circle method}\label{outline} 

To illustrate our main method, we focus on a particular case of Theorem \ref{thm:main1}, namely on proving boundedness of the maximal function $M_N^\chi$ on $\ell^2(\G_0)$. For simplicity of notation, for  $k\in \N$ and $x\in\G_0$, let
\begin{equation}\label{picu1}
\begin{split}
\M_kf(x)&:=M_{2^k}^{\chi}f(x)=\sum_{n\in\mathbb{Z}}2^{-k}\chi(2^{-k}n)f(A_0(n)^{-1}\cdot x)=(f\ast K_k)(x),\\
K_k(x)&:=G_{2^k}^{\chi}(x)=\sum_{n\in\Z}2^{-k}\chi(2^{-k}n)\ind{\{A_0(n)\}}(x),
\end{split}
\end{equation}
see \eqref{eq:47} and \eqref{eq:46} for the definitions $M_{N}^{\chi}$ and $G_{N}^{\chi}$ respectively. With this new notation, our main goal is to prove the following:

\begin{theorem}\label{picu2}
For any  $f\in\ell^2(\G_0)$ we have
\begin{align}
\label{eq:61}
\big\|\sup_{k\geq 0}|\M_kf|\big\|_{\ell^2(\G_0)}\lesssim \|f\|_{\ell^2(\G_0)}.
\end{align}
\end{theorem}

In the rest of this section we outline the proof of this theorem. Our main new construction is an iterative procedure, starting from the center of the group and moving down along its central series, that allows us to use some of the ideas of the classical circle method recursively at every stage. In our case of nilpotent groups of step two, the procedure consists of two basic stages and one additional step corresponding to ``major arcs". 

Notice that the kernels $K_k$ have product structure
\begin{equation}\label{picu5}
K_k(g):=L_k(g^{(1)})\ind{\{0\}}(g^{(2)}),\qquad L_k(g^{(1)}):=\sum_{n\in\Z}2^{-k}\chi(2^{-k}n)\ind{\{0\}}(g^{(1)}-A^{(1)}_0(n)),
\end{equation}
where $A^{(1)}_0(n):=(n,\ldots,n^d)\in\mathbb{Z}^d$ and $g=(g^{(1)},g^{(2)})\in\G_0$ as in \eqref{picu4}. 

\subsection{First stage reduction} We first decompose the singular kernel $\ind{\{0\}}(g^{(2)})$ in the central variable $g^{(2)}$ into smoother kernels. For any $s\in\NN$ and $m\in\Z_+$ we define the set of rational fractions 
\begin{equation}\label{picu6}
\mathcal{R}_s^m:=\{a/q:\,a=(a_1,\ldots, a_m)\in\Z^m,\,q\in[2^s,2^{s+1})\cap\Z,\,\mathrm{gcd}(a_1,\ldots,a_m,q)=1\}.
\end{equation}
We define also $\mathcal{R}^m_{\leq a}:=\bigcup_{0\le s\leq a}\mathcal{R}_s^m$. For $x^{(1)}=(x^{(1)}_{l_10})_{l_1\in\{1,\ldots,d\}}\in\R^{d}$, $x^{(2)}=(x^{(2)}_{l_1l_2})_{(l_1,l_2)\in Y'_d}\in\R^{d'}$ and $\Lambda\in(0,\infty)$ we define the partial dilations
\begin{equation}\label{tra3.55}
\Lambda\circ x^{(1)}=(\Lambda^{l_1}x^{(1)}_{l_10})_{l_1\in\{1,\ldots,d\}}\in \R^{d},\qquad\Lambda\circ x^{(2)}=(\Lambda^{l_1+l_2}x^{(2)}_{l_1l_2})_{(l_1,l_2)\in Y'_d}\in \R^{d'},
\end{equation}
which are induced by the group-dilations defined in \eqref{tra3.5}. 

We fix $\eta_0:\mathbb{R}\to[0,1]$ a smooth even function
such that $\ind{[-1, 1]}\le \eta_0\le \ind{[-2, 2]}$. For $t\in\mathbb{R}$ and integers $j\geq 1$ we define
\begin{equation}\label{cutR}
\eta_j(t):=\eta_0(2^{-j}t)-\eta_0(2^{-j+1}t),\qquad 1=\sum_{j=0}^\infty \eta_j .
\end{equation}
For any $A\in[0,\infty)$ we define
\begin{equation}\label{cutR2}
\eta_{\leq A}:=\sum_{j\in[0,A]\cap \Z}\eta_j.
\end{equation}
By a slight abuse of notation we also let $\eta_j$ and $\eta_{\leq A}$ denote the smooth radial functions on $\R^m$, $m\geq 1$, defined by $\eta_j(x)=\eta_j(|x|)$ and $\eta_{\leq A}(x)=\eta_{\leq A}(|x|)$. We fix also two small constants $\delta=\delta(d)\ll\delta'=\delta'(d)$ such that $\delta'\in(0,(10d)^{-10}]$ and $\delta\in(0,{(\delta')}^4]$, and a large constant $D=D(d)\gg\delta^{-8}$, which depend on arithmetic properties of the polynomial sequence $A_0$ (more precisely on the structural constants in Propositions \ref{minarcs}--\ref{minarcscon}) such that
\begin{equation}\label{ConstantsStr}
1\ll 1/\delta'\ll 1/\delta\ll r=r(\delta, \delta',d)\ll D.
\end{equation}

For $k\geq D^2$ we fix two cutoff functions $\phi_k^{(1)}:\R^{d}\to[0,1]$, $\phi_k^{(2)}:\R^{d'}\to[0,1]$, such that
\begin{equation}\label{picu6.6}
\phi_k^{(1)}(g^{(1)}):=\eta_{\leq\delta k}(2^{-k}\circ g^{(1)}),\qquad\phi_k^{(2)}(g^{(2)}):=\eta_{\leq\delta k}(2^{-k}\circ g^{(2)}).
\end{equation}
For $k\in \NN$ so that $k\geq D^2$ and for any $1$-periodic sets of rationals $\A \subseteq \mathbb{Q}^d$, $\B \subseteq \mathbb{Q}^{d'}$
we define the periodic Fourier multipliers by 
\begin{equation} \label{def:progen}
\begin{split}
\Psi_{k,\A}(\xi^{(1)})
&:=\sum_{a/q\in \A}\eta_{\leq\delta' k}(2^k\circ(\xi^{(1)}-a/q)), 
\qquad \xi^{(1)} \in \mathbb{T}^{d}, 
\\ 
\Xi_{k,\B}(\xi^{(2)})
& :=\sum_{b/q\in \B} \eta_{\leq\delta k}(2^{k}\circ(\xi^{(2)}-b/q)), 
\qquad \xi^{(2)} \in \mathbb{T}^{d'}.
\end{split}
\end{equation}
For $k\geq D^2$ and $s\in[0,\delta k]\cap\Z$ we define the periodic Fourier multipliers $\Xi_{k,s}:\R^{d'}\to [0,1]$,
\begin{equation}\label{picu7}
\Xi_{k,s}(\xi^{(2)}):=\Xi_{k,\mathcal{R}_s^{d'}}(\xi^{(2)})=\sum_{a/q\in\mathcal{R}_s^{d'}}\eta_{\leq\delta k}(2^{k}\circ(\xi^{(2)}-a/q)).
\end{equation}

For $k\geq D^2$ we write
\begin{equation}\label{picu8}
\begin{split}
\ind{\{0\}}(g^{(2)})&=\int_{\T^{d'}}\ex(g^{(2)}.\xi^{(2)})\, d\xi^{(2)}\\
&=\sum_{s\in[0,\delta k]\cap\Z}\int_{\T^{d'}}\ex(g^{(2)}.\xi^{(2)})\Xi_{k,s}(\xi^{(2)})\, d\xi^{(2)}+\int_{\T^{d'}}\ex(g^{(2)}.\xi^{(2)})\Xi_k^c(\xi^{(2)})\, d\xi^{(2)},
\end{split}
\end{equation}
recall that $g^{(2)}.\xi^{(2)}$ denotes the usual scalar product of vectors in $\R^{d'}$ and
\begin{equation}\label{picu9}
\Xi_k^c:=1-\sum_{s\in[0,\delta k]\cap\Z}\Xi_{k,s}.
\end{equation}
Then we decompose $K_k=K_k^c+\sum_{s\in[0,\delta k]\cap\Z}K_{k,s}$, where, with the notation in \eqref{picu5}, we have
\begin{equation}\label{picu10}
K_{k,s}(g):=L_k(g^{(1)})N_{k,s}(g^{(2)}),\qquad K_{k}^c(g):=L_k(g^{(1)})N_k^c(g^{(2)}),
\end{equation}
and
\begin{equation}\label{picu10.2}
\begin{split}
N_{k,s}(g^{(2)})&:=\phi_k^{(2)}(g^{(2)})\int_{\T^{d'}}\ex(g^{(2)}.\xi^{(2)})\Xi_{k,s}(\xi^{(2)})\, d\xi^{(2)},\\ N_k^c(g^{(2)})&:=\phi_k^{(2)}(g^{(2)})\int_{\T^{d'}}\ex(g^{(2)}.\xi^{(2)})\Xi_{k}^c(\xi^{(2)})\, d\xi^{(2)}.
\end{split}
\end{equation}

We first show that we can bound the contributions of the minor arcs in
the central variables:

\begin{lemma}\label{MinArc1}
For any integer $k\geq D^2$ and $f\in\ell^2(\G_0)$ we have
\begin{equation}\label{picu12}
\|f\ast K_k^c\|_{\ell^2(\G_0)}\lesssim 2^{-k/D^2}\|f\|_{\ell^2(\G_0)}.
\end{equation}
\end{lemma}

Then we prove our first transition estimate, i.e. we show that we can bound the
contributions of the kernels $K_{k,s}$ corresponding to scales $k\ge0$ not very large. More precisely, for any $s\geq 0$ we define
\begin{equation}\label{eq:1}
\kappa_s:=2^{2D(s+1)^2}.
\end{equation}

\begin{lemma}\label{MajArc1}
For any integer $s\ge0$  and  $f\in\ell^2(\G_0)$ we have
\begin{equation}\label{picu13}
\big\|\sup_{\max(D^2,s/\delta)\leq k< 2\kappa_s}|f\ast K_{k,s}|\big\|_{\ell^2(\G_0)}\lesssim 2^{-s/D^2}\|f\|_{\ell^2(\G_0)}.
\end{equation}
\end{lemma}

In the commutative setting, minor arcs estimates such as \eqref{picu12} follow using Weyl estimates and the Plancherel theorem. As we do not have a useful Fourier transform on the group $\G_0$, our main tool to prove the bounds \eqref{picu12} is a high order $T^\ast T$ argument. More precisely, we analyze the kernel of the convolution operator $\{(\mathcal{K}_k^c)^\ast\mathcal{K}_k^c\}^r$, where $\mathcal{K}_{k}^cf:=f\ast K_k^c$ and $r$ is sufficiently large, and show that its $\ell^1(\G_0)$ norm  is $\lesssim 2^{-k}$. The main ingredient in this proof is the non-commutative Weyl estimate in Proposition \ref{minarcs} (i).

To prove the transition estimates \eqref{picu13}, we use the Rademacher-Menshov inequality and Khintchine's ine\-qua\-lity (leading to logarithmic losses) to reduce to proving the bounds
\begin{equation}\label{illust9.5}
\Big\|\sum_{k\in[J,2J]}\varkappa_k(f\ast H_{k,s})\Big\|_{\ell^2(\G_0)}\lesssim 2^{-4s/D^2}\big\|f\big\|_{\ell^2(\G_0)}
\end{equation}
for any $J \ge \max(D^2,s/\delta)$ and any coefficients $\varkappa_k\in[-1,1]$, where $H_{k,s}:=K_{k+1,s}-K_{k,s}$. For this, we use a high order version of the Cotlar--Stein lemma, which relies again on precise analysis of the kernel of the convolution operator $\{(\mathcal{H}_{k,s})^\ast\mathcal{H}_{k,s}\}^r$, where $\mathcal{H}_{k,s}f:=f\ast H_{k,s}$ and $r$ is sufficiently large. The key exponential gain of $2^{-4s/D^2}$ in \eqref{illust9.5} is due to the non-commutative Gauss sums estimate, see Proposition \ref{minarcs} (ii).

\subsection{Second stage reduction} In view of Lemmas \ref{MinArc1}--\ref{MajArc1} it remains to prove that
\begin{equation}\label{illust10}
\big\|\sup_{k\geq \kappa_s}|f\ast K_{k,s}|\big\|_{\ell^2(\G_0)}\lesssim 2^{-s/D^2}\|f\|_{\ell^2(\G_0)}
\end{equation}
for any fixed integer $s\geq 0$. The kernels $K_{k,s}$ are now reasonably well adapted to a natural family of non-isotropic balls in the central variables, at least when $2^s\approx 1$, and we need to start decomposing in the non-central variables.

We examine the kernels $L_k(g^{(1)})$ defined in \eqref{picu5}, and rewrite them in the form
\begin{equation}\label{picu19}
\begin{split}
L_k(g^{(1)})&=\sum_{n\in\Z}2^{-k}\chi(2^{-k}n)\ind{\{0\}}(-A^{(1)}_0(n)+g^{(1)})\\
&=\phi_k^{(1)}(g^{(1)})\int_{\T^d}\ex(g^{(1)}{.}\xi^{(1)})S_k(\xi^{(1)})\,d\xi^{(1)},
\end{split}
\end{equation}
where $g^{(1)}.\xi^{(1)}$ denotes the usual scalar product of vectors in $\R^{d}$, and 
\begin{align}
\label{picu18}
\begin{split}
S_k(\xi^{(1)})&:=\sum_{n\in\Z}2^{-k}\chi(2^{-k}n)\ex(-A^{(1)}_0(n){.}\xi^{(1)}).
\end{split}
\end{align}
For any integers $Q\in\Z_+$ and $m\in\Z_+$ we define the set of fractions
\begin{equation}\label{picu20}
\widetilde{\mathcal{R}}^m_{Q}:=\{a/Q:\,a=(a_1,\ldots,a_m)\in\Z^m\}.
\end{equation}
For any integer $s\ge0$ we fix a large denominator
\begin{equation}
\label{eq:2}
Q_s:=\big(\big\lfloor 2^{D(s+1)}\big\rfloor\big)!=1\cdot 2\cdot\ldots\cdot \big\lfloor 2^{D(s+1)}\big\rfloor,
\end{equation}
and using \eqref{def:progen} define the periodic multipliers
\begin{equation}\label{picu21}
\begin{split}
\Psi_{k,s}^{\low}(\xi^{(1)})&:=\Psi_{k,\widetilde{\mathcal{R}}^d_{Q_s}}(\xi^{(1)})
=\sum_{a/q\in\widetilde{\mathcal{R}}^d_{Q_s}}\eta_{\leq\delta' k}(2^k\circ(\xi^{(1)}-a/q)),\\
\Psi_{k,s,t}(\xi^{(1)})&:=\Psi_{k,\mathcal{R}^d_t\setminus \widetilde{\mathcal{R}}^d_{Q_s}}(\xi^{(1)})
=\sum_{a/q\in\mathcal{R}^d_t\setminus \widetilde{\mathcal{R}}^d_{Q_s}}\eta_{\leq\delta' k}(2^k\circ(\xi^{(1)}-a/q)),\\
\Psi_{k}^c(\xi^{(1)})&:=1-\Psi_{k,s}^{\low} (\xi^{(1)}) -\sum_{t\in[0,\delta'k]\cap\Z}\Psi_{k,s,t} (\xi^{(1)}) 
\ =1-\sum_{a/q\in\mathcal{R}^d_{\leq\delta' k}}\eta_{\leq\delta' k}(\tau^k\circ(\xi^{(1)}-a/q)).
\end{split}
\end{equation}
Since $k\geq \kappa_s=2^{2D(s+1)^2}$ we see that $Q_s\leq 2^{\delta^2 k}$. Therefore the supports of the cutoff functions $\eta_{\leq\delta' k}(2^k\circ(\xi^{(1)}-a/q))$ are all disjoint and the multipliers $\Psi_{k,s}^{\low}, \Psi_{k,s,t}, \Psi_{k}^c$ take values in the interval $[0,1]$. Notice also that $\Psi_{k,s,t}\equiv 0$ unless $t\geq D(s+1)$, and that the cutoffs used in these definitions depend on $\delta' k$ not on $\delta k$ as in the case of the central variables.

We examine the formula \eqref{picu19} and define the kernels $L_{k,s}^{\low}, L_{k,s,t}, L_{k}^c:\Z^d\to\C$ by
\begin{equation}\label{picu24}
L_\ast(g^{(1)})=\phi_k^{(1)}(g^{(1)})\int_{\T^d}\ex(g^{(1)}{.}\xi^{(1)})S_k(\xi^{(1)})\Psi_\ast(\xi^{(1)})\,d\xi^{(1)},
\end{equation}
where  $(L_\ast, \Psi_\ast)\in\{(L_{k,s}^{\low}, \Psi_{k,s}^{\low}), (L_{k,s,t}, \Psi_{k,s,t}), (L_{k}^c, \Psi_{k}^c)\}$. For any $k\geq \kappa_s$  we
obtain $K_{k, s}=G_{k,s}^{\low}+\sum_{t\leq\delta' k}G_{k,s,t}+G_{k,s}^c$, where 
the kernels $G_{k,s}^{\low}, G_{k,s,t}, G_{k,s}^c:\Z^{|Y_d|}\to\C$ are defined by
\begin{equation}\label{picu25}
\begin{split}
G_{k,s}^{\low}(g)&:=L_{k,s}^{\low}(g^{(1)})N_{k,s}(g^{(2)}),\\
G_{k,s,t}(g)&:=L_{k,s,t}(g^{(1)})N_{k,s}(g^{(2)}),\\
G_{k,s}^c(g)&:=L_{k}^c(g^{(1)})N_{k,s}(g^{(2)}).
\end{split}
\end{equation}

Our next step is to show that the contributions of the minor arcs corresponding to the kernels $G_{k,s}^c$ can be suitably bounded:

\begin{lemma}\label{MinArc2}
For any integers $s\geq 0$ and $k\geq \kappa_s $,  and for any $f\in\ell^2(\G_0)$ we have
\begin{equation}\label{picu26}
\|f\ast G_{k,s}^c\|_{\ell^2(\G_0)}\lesssim  2^{-k/D^2}\|f\|_{\ell^2(\G_0)}.
\end{equation}
\end{lemma}

Then we prove our second transition estimate, bounding the contributions of the operators defined by the kernels $G_{k,s,t}$ for intermediate values of $k$.

\begin{lemma}\label{MajArc3}
For any integers  $s\ge 0$, and $t\geq D(s+1)$, and $f\in\ell^2(\G_0)$ we have
\begin{equation}\label{picu28}
\big\|\sup_{\max(\kappa_s,t/\delta')\leq k< 2\kappa_t}|f\ast G_{k,s,t}|\big\|_{\ell^2(\G_0)}\lesssim 2^{-t/D^2}\|f\|_{\ell^2(\G_0)},
\end{equation}
where $\kappa_t=2^{2D(t+1)^2}$ as in \eqref{eq:1}.
\end{lemma}

The proofs of these estimates are similar to the proofs of the corresponding first stage estimates \eqref{picu12} and \eqref{picu13}, using high order $T^\ast T$ arguments. However, instead of using the nilpotent oscillatory sums estimates in Proposition \ref{minarcs}, we use the classical estimates from Proposition \ref{minarcscom} here. We emphasize, however, that the underlying nilpotent structure is very important and that these estimates are only possible after performing the two reductions in the first stage, namely, the restriction to major arcs corresponding to denominators $\approx 2^s$ and the restriction to parameters $k\geq \kappa_s$.
We finally remark that the circle method could not have been applied simultaneously to both central and non-central variables, as we would not have been able control efficiently the phase functions arising in the corresponding exponential sums and oscillatory integrals, especially on major arcs.

\subsection{Final stage: major arcs contributions} After these reductions, it remains to bound the contributions of the ``major arcs" in both the central and the non-central variables. More precisely, we prove the following bounds:
\begin{lemma}\label{MajArc2} (i) For any integer $s\ge 0$ and $f\in\ell^2(\G_0)$ we have
\begin{equation}\label{picu27}
\big\|\sup_{k\geq \kappa_s}|f\ast G^{\low}_{k,s}|\big\|_{\ell^2(\G_0)}\lesssim  2^{-s/D^2}\|f\|_{\ell^2(\G_0)}.
\end{equation}

(ii) For any integers $s\ge 0$, $t\geq D(s+1)$, and $f\in\ell^2(\G_0)$ we have
\begin{equation}\label{picu27.5}
\big\|\sup_{k\geq \kappa_t}|f\ast G_{k,s,t}|\big\|_{\ell^2(\G_0)}\lesssim 2^{-t/D^2}\|f\|_{\ell^2(\G_0)}.
\end{equation}
\end{lemma}

The main idea here is different: we write the kernels $G^{\rm low}_{k,s}$ and $G_{k,s,t}$ as tensor products of two components up to acceptable errors. One of these components is essentially a maximal average operator on a continuous group, which can be analyzed using the classical method of Christ \cite{Ch}. The other component is an arithmetic operator-valued analogue of the classical Gauss sums, which leads to the key factors $2^{-s/D^2}$ and $2^{-t/D^2}$ in \eqref{picu27} and \eqref{picu27.5}.

More precisely, for any integer $Q\geq 1$ we define the subgroup
\begin{equation}\label{gio1}
\HH_Q:=\{h=(Qh_{l_1l_2})_{(l_1,l_2)\in Y_d}\in\G_0:\,h_{l_1,l_2}\in\Z\}.
\end{equation}
Clearly $\HH_Q\subseteq\G_0$ is a normal subgroup. Let $\JJ_Q$ denote the coset
\begin{equation}\label{gio2}
\JJ_Q:=\{b=(b_{l_1l_2})_{(l_1,l_2)\in Y_d}\in\G_0:\,b_{l_1,l_2}\in\Z\cap[0,Q-1]\},
\end{equation}
with the natural induced group structure. Notice that 
\begin{equation}\label{gio3}
\text{ the map }(b,h)\mapsto b\cdot h\text{ defines a bijection from }\JJ_Q\times\HH_Q\text{ to }\G_0.
\end{equation}
Assume that $Q\geq 1$ and $2^k\geq Q$. For any $a\in\Z^d$  and $\xi\in\R^d$ let
\begin{equation}\label{gio3.5}
\begin{split}
&J_k(\xi):=2^{-k}\int_{\R}\chi(2^{-k}x)\ex[-A_0^{(1)}(x){.}\xi]\,dx
=\int_{\R}\chi(y)\ex[-A_0^{(1)}(y){.}(2^k\circ\xi)]\,dy,\\
&S(a/Q):=Q^{-1}\sum_{n\in \Z_Q}\ex[-A_0^{(1)}(n){.}a/Q].
\end{split}
\end{equation}

The point is that the kernels $G^{\rm low}_{k,s}$ and $G_{k,s,t}$ can be decomposed as tensor products. Indeed, to decompose $G_{k,s,t}$ (the harder case) we set $Q:=Q_t=\big(\big\lfloor 2^{D(t+1)}\big\rfloor\big)!$ as in \eqref{eq:2}. Then we show that if $k\geq \kappa_t$ (so $2^k\gg Q_t^4$), $h\in\HH_{Q_t}$ and $b_1,b_2\in\G_0$ satisfy $|b_1|+|b_2|\leq Q^4$ then
\begin{equation}\label{kio6}
G_{k,s,t}(b_1\cdot h\cdot b_2)\approx W_{k,Q_t}(h)V_{\mathcal{R}^d_t\setminus \widetilde{\mathcal{R}}^d_{Q_s},\mathcal{R}_s^{d'}, Q_t}(b_1\cdot b_2),
\end{equation}
up to acceptable summable errors. Here
\begin{equation*}
W_{k,Q}(h):=Q^{d+d'}\phi_k(h)\int_{\R^d\times\R^{d'}}\eta_{\leq\delta' k}(2^k\circ\xi)\eta_{\leq\delta k}(2^k\circ\theta)\ex(h{.}(\xi,\theta))J_k(\xi)\,d\xi d\theta,
\end{equation*}
\begin{equation*}
V_{\A, \B,Q}(b)
:=Q^{-d-d'}\Big\{\sum_{\sigma^{(1)}\in\A\cap[0,1)^d}S(\sigma^{(1)})\ex[b^{(1)}{.}(\sigma^{(1)})]\Big\}
\Big\{\sum_{\sigma^{(2)}\in\B\cap[0,1)^{d'}}\ex[b^{(2)}.(\sigma^{(2)})]\Big\},
\end{equation*}
and $\phi_k(h):=\phi_{k}^{(1)}(h^{(1)})\phi_{k}^{(2)}(h^{(2)})$, $h=(h^{(1)},h^{(2)})\in\HH_{Q}$, is defined in \eqref{picu6.6}, $b=(b^{(1)},b^{(2)})\in\G_0$, and the functions $J_k$ and $S$ are defined in \eqref{gio3.5}.

Finally, we show that the kernels $V_{s,t}:=V_{\mathcal{R}^d_t\setminus \widetilde{\mathcal{R}}^d_{Q_s},\mathcal{R}_s^{d'}, Q_t}$ (which can be interpreted as an operator-valued Gauss sums) define bounded operators on $\ell^2(\JJ_{Q_t})$, 
\begin{equation*}
\|f\ast_{\JJ_{Q_t}} V_{s,t}\|_{\ell^2(\JJ_{Q_t})}\lesssim 2^{-t/D}\|f\|_{\ell^2(\JJ_{Q_t})}.
\end{equation*}
Moreover, the kernels $W_{k,Q_t}$ are close to classical maximal operators and one can show that
\begin{equation*}
\big\|\sup_{k\geq \kappa_t}|f\ast_{\HH_{Q_t}} W_{k,Q_t}|\big\|_{\ell^2(\HH_{Q_t})}\lesssim \|f\|_{\ell^2(\HH_{Q_t})}.
\end{equation*}
The desired bounds \eqref{picu27.5} follow using the approximation formula \eqref{kio6}.

\section{A nilpotent Waring theorem on the group $\G_0$}\label{notation}

The classical Waring problem, solved by Hilbert \cite{Hilb} in 1909, concerns the possibility of writing any positive integer as a sum of finitely many $p$ powers: for any integer $p\geq 1$ there is $N(p)$ such that any integer $y\in\Z_+$ can be written in the form
\begin{equation}\label{Wri1}
y=\sum_{i=1}^{N(p)}x_i^p,\qquad \text{ for some non-negative integers }x_1,\ldots,x_{N(p)}.
\end{equation}

There is a vast amount of literature on this problem and its many possible extensions. We are interested here in understanding the analogous question on our discrete nilpotent Lie group $\G_0$ and for our given polynomial sequence $A_0$: can one represent elements $g\in\G_0$ in the form 
\begin{equation}\label{Wri2}
g
=
A_0(n_1)^{-1}\cdot A_0(m_1)\cdot\ldots\cdot A_0(n_r)^{-1}\cdot A_0(m_r),
\end{equation}
for some integers $n_1,m_1,\ldots,n_r,m_r$, provided that $r$ is large enough? We are, in fact, interested in proving a quantitative statement on the number of such representations, for integers $n_1,m_1,\ldots,n_r,m_r\in [N]:=[-N,N]\cap\Z$.

We make two observations. First, many group elements $g$ cannot be written in the form \eqref{Wri2}, due to local obstructions; for instance, if $g$ can be represented in the form \eqref{Wri2} then necessarily $g_{10}\equiv g_{20}\equiv\ldots\equiv g_{d0}\, (\mathrm{mod}\,2)$, $g_{10}=g_{30}=\ldots \, (\mathrm{mod}\,3)$ etc. Second, there is a significant difference between the classical Waring problem \eqref{Wri1} and its nilpotent analogue \eqref{Wri2}, namely the positivity of the $p$-powers which imposes size restrictions on the variables $x_i$ in terms of the prescribed output value $y$.

For integers $r,N\geq 1$ and $g\in\G_0$ let 
\begin{align} \label{id:2}
 S_{r,N}(g):=\big|\big\{ (m,n) \in [N]^{2r}:\, A_0(n_1)^{-1}\cdot A_0(m_1)\cdot\ldots\cdot A_0(n_r)^{-1}\cdot A_0(m_r) = g  \big\}\big|.
\end{align}
Our main result in this section is the following:

\begin{theorem} \label{thm:1}

(i) There is an integer $r_0(d)\geq 1$ such that if $r\geq r_0(d)$ is sufficiently large and $g\in\G_0$ then
	\begin{align} \label{id:asym}
		S_{r,N}(g) 
		& = N^{2r}
		 \Big( \prod_{(l_1,l_2) \in Y_d} N^{-\abs{l_1} - \abs{l_2}} \Big)
		\Big[\mathfrak{S} (g)  \int_{\R^{d+d'}} \Phi (\zeta) \ex(-(N^{-1}\circ g) . \zeta)\, d\zeta+O_{r} (N^{-1/2})\Big],
	\end{align}
uniformly in $N \in \N$. Here the singular series $\mathfrak{S}$ is defined by
\begin{align} \label{id:SS}
	\mathfrak{S} (g)
	:=
	\sum_{a/q\in \mathcal{R}_{\infty}^{d+d'} \cap[0,1)^{d+d'}} \overline{G(a/q)} \ex(-g . a/q)
\end{align}
and the singular integral $\Phi$ is defined by
\begin{align} \label{def:Phi}
	\Phi(\xi) = \int_{[-1,1]^{2r}} \ex\big( D(z,w) . \xi \big) \, dz \, dw, \qquad \xi \in \R^{d+d'}.
\end{align}
In particular, all elements $g\in\G_0$ cannot be represented in the form \eqref{Wri2} more than a constant times the expected number of representations, i.e.
\begin{equation}\label{Wri1.5}
S_{r,N}(g)\lesssim_r N^{2r}\Big( \prod_{(l_1,l_2) \in Y_d} N^{-\abs{l_1} - \abs{l_2}} \Big)\qquad\text{ for any }g\in\G_0.
\end{equation}

(ii) For $r_0(d)$ as above, if $r\geq r_0(d)$ and $r$ is even, then there is a sufficiently large integer $Q=Q(r)$ and an element $g_0\in\JJ_Q$ (see definitions \eqref{gio1}--\eqref{gio2}) such that
\begin{equation}\label{Wri2.5}
S_{r,N}(g) = N^{2r}\Big( \prod_{(l_1,l_2) \in Y_d} N^{-\abs{l_1} - \abs{l_2}} \Big)\Big[c_r(g)+O_{r,g} (N^{-1/2})\Big],
\end{equation}
for any $g\in g_0\HH_Q$, where $c_r(g)\approx_r 1$ uniformly in $g$.
\end{theorem}

\begin{proof} Observe that $D(n,m) = A_0(n_1)^{-1}\cdot A_0(m_1)\cdot\ldots\cdot A_0(n_r)^{-1}\cdot A_0(m_r)$, see \eqref{pro0.3.5}. Using the classical delta function we can write
\begin{equation}\label{Wri5}
	S_{r,N}(g) =\sum_{m,n \in [N]^{r}} \int_{\T^{d+d'}} \ex\big( D(n,m) . \xi \big) \ex(-g . \xi) \, d\xi.
\end{equation}

\smallskip
{\bf{Step 1.}} We start by decomposing the integration in $\xi$ into major and minor arcs.
For any integer $m\geq 1$ and any positive number $M > 0$, we define the set of rational fractions 
\begin{equation}\label{id:3}
	\mathcal{R}_{\le M}^m:=\{a/q:\,a=(a_1,\ldots, a_m)\in\Z^m,\, q \in [1,M]\cap\Z,\,\mathrm{gcd}(a_1,\ldots,a_m,q)=1\}.
\end{equation}
Notice that we use a bit different definition of $\mathcal{R}_{\le M}^m$ than in \eqref{picu6}.
We fix a small constant $\delta=\delta(d)\ll 1$ and a smooth radial function $\eta_0 \colon \R^{|Y_d|} \to [0,1]$ such that $\ind{|x|\le 1}\le \eta_0 (x) \le \ind{|x|\le 2}$, $x \in \R^{|Y_d|}$. For $A>0$ let $\eta_{\leq A} (x):= \eta_0 \big(A^{-1} x \big)$, $x\in \R^{|Y_d|}$; 
here we use a bit different definition of $\eta_{\leq A}$ than in \eqref{cutR2}.
Then we introduce the projections
\begin{align} \label{id:5}
	\Xi_{N}(\xi)
	:=\sum_{a/q\in \mathcal{R}_{\le N^\delta}^{d+d'}} 
	\eta_{\leq N^\delta} \big(N \circ (\xi-a/q) \big), 
	\qquad \xi \in \mathbb{T}^{d+d'}, \quad N \in \N,
\end{align}
and decompose the integration in \eqref{Wri5} into major and minor arcs, i.e. we define 
\begin{align} \label{id:maj}
	S_{r,N, \maj}(g) &:= \sum_{m,n \in [N]^{r}} \int_{\T^{d+d'}} \ex\big( D(n,m) . \xi \big) \ex(-g . \xi) \Xi_{N}(\xi) \, d\xi, \\ \label{id:min}
	S_{r,N, \minor}(g) &:= 
	\sum_{m,n \in [N]^{r}} \int_{\T^{d+d'}} \ex\big( D(n,m) . \xi \big) \ex(-g . \xi) \big(1 - \Xi_{N}(\xi) \big) \, d\xi,
\end{align}
Notice that $S_{r,N}(g) = S_{r,N, \minor}(g) + S_{r,N, \maj}(g)$. Moreover
\begin{align}\label{Wri4}
		\abs{S_{r,N, \minor}(g)} \lesssim_{r} N^{2r - 1}\Big( \prod_{(l_1,l_2) \in Y_d} N^{-\abs{l_1} - \abs{l_2}} \Big), \qquad N \in \N, \quad g \in \G_0
		\end{align}
provided that $r$ is sufficiently large, as a consequence of Proposition \ref{minarcs} (i) and the Dirichlet principle; in fact we use Proposition \ref{minarcs} (i) with $\phi_N^{(j)} = \psi_N^{(j)} = \ind{[N]}$, $1 \le j \le r$, which is still valid as can be seen by careful reading of the proof of this result contained in \cite{IMW}. Therefore the contribution of the minor arcs $S_{r,N, \minor}(g)$ can be absorbed by the error term in \eqref{id:asym}. 

\smallskip
{\bf{Step 2.}} Next, we deal with the major arcs contributions. Notice that 
\begin{align} \label{id:7}
	S_{r,N, \maj}(g) 
	= 
	\sum_{a/q\in \mathcal{R}_{\le N^\delta}^{d+d'}\cap[0,1)^{d+d'}} \ex(-g . a/q) 
	 \int_{\R^{d+d'}} \eta_{\leq N^\delta} \big(N \circ \xi \big) I_{r,N, a/q}(\xi) 
	 \ex(-g . \xi)  \, d\xi, 
\end{align}
where
\begin{align} \label{id:8}
	I_{r,N, a/q}(\xi) 
	= 
	\sum_{m,n \in [N]^{r}} \ex\big( D(n,m) . (a/q) \big)  \ex\big( D(n,m) . \xi \big).
\end{align}
Observe that for $a/q\in \mathcal{R}_{\le N^\delta}^{d+d'}\cap[0,1)^{d+d'}$ and 
$\abs{N \circ \xi} \lesssim N^\delta$ we have 
\begin{align*}
	I_{r,N, a/q}(\xi) 
	& = 
	\sum_{m,n \in [N/q]^{r}} \sum_{u,v\in \Z_q^r} \ex\big( D(v,w){.}(a/q) \big)
	   \ex\big( D(qn,qm) . \xi \big) 
	 + O(q N^{2r-1 + \delta}) \\
	 & =
	 N^{2r} \overline{G(a/q)} \Phi(N \circ \xi)+ O(q N^{2r-1 + \delta}),
\end{align*}
where $G(a/q)$ is defined in \eqref{pro0.6} and $\Phi$ is defined in \eqref{def:Phi}.

Therefore, if $\delta\leq (10d)^{-4}$ then we have
\begin{equation}\label{Wri7} 
\begin{split}
&S_{r,N, \maj}(g)=N^{2r} \Big( \prod_{(l_1,l_2) \in Y_d} N^{-\abs{l_1} - \abs{l_2}} \Big)\\
	&\quad\times\Big[\sum_{a/q\in \mathcal{R}_{\le N^\delta}^{d+d'}} \overline{G(a/q)} \ex(-g . a/q) 
	\int_{\R^{d+d'}} \eta_{\leq N^\delta} \big(\xi \big) \Phi(\xi)
	\ex(-g . (N^{-1}\circ\xi))  \, d\xi+O_r(N^{-1/2})\Big].
\end{split}
\end{equation}
It follows from Proposition~\ref{minarcs} (ii) and Proposition \ref{minarcscon} that
\begin{align} \label{estGauss}
	\abs{G(a/q)} \lesssim_r q^{-1/\delta^2},\qquad (a,q)=1,
\end{align}
and
\begin{align} \label{id:9}
	\abs{\Phi (\zeta)} \lesssim_r \langle \zeta \rangle^{-1/\delta^2}, \qquad \zeta \in \R^{d+d'},
\end{align}
provided that $r$ is sufficiently large. Therefore, recalling the definition \eqref{id:SS},
\begin{equation}\label{Wri8}
\begin{split}
&|\mathfrak{S} (g)|\lesssim_r 1,\\
&\abs[\Big]{\mathfrak{S} (g) - \sum_{a/q\in \mathcal{R}_{\le N^\delta}^{d+d'}} \overline{G(a/q)} \ex(-g . a/q)}\lesssim_r\sum_{q \ge N^{\delta}} q^{d + d' - 1/\delta^2}\lesssim_r N^{-1/(2\delta)}.
\end{split}
\end{equation}
Moreover, we have
\begin{equation}\label{Wri9}
\begin{split}
&\Big|\int_{\R^{d+d'}} \eta_{\leq N^\delta} \big(\xi \big) \Phi(\xi)\ex(-g . (N^{-1}\circ\xi))  \, d\xi\Big|\lesssim_r 1,\\
&\Big|\int_{\R^{d+d'}} \eta_{\leq N^\delta} \big(\xi \big) \Phi(\xi)\ex(-g . (N^{-1}\circ\xi))  \, d\xi-\int_{\R^{d+d'}} \Phi(\xi)\ex(-g . (N^{-1}\circ\xi))  \, d\xi\Big|\lesssim_r N^{-1/(2\delta)}.
\end{split}
\end{equation}
It follows from \eqref{Wri7}, \eqref{Wri8}, and \eqref{Wri9} that
\begin{equation}  \label{id:11}
\begin{split}
&S_{r,N, \maj}(g)=N^{2r} \Big( \prod_{(l_1,l_2) \in Y_d} N^{-\abs{l_1} - \abs{l_2}} \Big)\Big[\mathfrak{S}(g)\int_{\R^{d+d'}}\Phi(\xi)
	\ex(-g . (N^{-1}\circ\xi))  \, d\xi+O_r(N^{-1/2})\Big].
\end{split}
\end{equation}
The desired conclusion \eqref{id:asym} follows using also \eqref{Wri4}. This completes the proof of part (i) of the theorem.

\smallskip
{\bf{Step 3.}} We analyze now the singular series $\mathfrak{S}$ defined in \eqref{id:SS}. Observe that 
\begin{equation}\label{Wri10}
\mathfrak{S} (h) = \sum_{q \ge 1} A(q,h),\qquad A(q,h) := \sum_{(a,q) = 1} \overline{G(a/q)} \ex(-h . a/q),	
\end{equation}
for any $h\in\G_0$. Notice that the sequence $A(q,h)$ is multiplicative in the sense that $A(q_1 q_2,h) = A(q_1,h) A(q_2,h)$ provided that  $(q_1,q_2)=1$ and $h\in\G_0$. Therefore, letting $\PP$ denote the set of primes,
\begin{align} \label{id:12}
\mathfrak{S} (h) = \prod_{p \in \PP} B(p,h), \qquad B(p,h):=1 + \sum_{n \ge 1} A(p^n,h).
\end{align}	

For $h\in\G_0$ and $q\geq 1$ let
\begin{align}\label{Wri17} 
M(q,h):= \big|\big\{ (m,n) \in\Z_q^{2r} : D(n,m) = h \, (\mathrm{mod}\,q)  \big\}\big|.
\end{align}
We prove that for any $h\in\G_0$, $p\in\PP$ and integer $n\geq 1$ we have
\begin{align} \label{5.1}
1 + \sum_{v=1}^n A(p^v,h) = \frac{M(p^n,h)}{p^{n(2r-d-d')}}.
\end{align}	
Indeed, for any integer $q\geq 1$ we have
\begin{align*} 
	M(q,h) 
	& = 
	q^{-d-d'} \sum_{t \in \Z_q^{d+d'}} \sum_{m, n \in \Z_q^{r}} \ex\big( (D(n,m) - h). (t/q) \big) \\
	& = 
	q^{-d-d'} \sum_{q_1|q} \sum_{w \in\Z_{q/q_1}^{d+d'},\,(w,q/q_1) = 1} 
	\sum_{m, n \in \Z_q^{r}}
	\ex\big( (D(n,m) - h). (wq_1/q) \big) \\
	& = 
	q^{-d-d'} \sum_{q_2|q} \sum_{w \in\Z_{q_2}^{d+d'},\,(w,q_2) = 1} 
	\sum_{m, n \in \Z_q^{r}}
	\ex\big( (D(n,m) - h). (w/q_2) \big) \\
	& = 
	q^{-d-d'} \sum_{q_2|q} \sum_{w \in \Z_{q_2}^{d+d'},\,(w,q_2) = 1} 
	q^{2r} \overline{G(w/q_2)} \ex\big(-h . (w/q_2) \big) \\
	& = 
	q^{2r -d-d'} \sum_{q_2|q} A(q_2,h).
\end{align*}	
The identity \eqref{5.1} follows by applying this with $q=p^n$, $p\in\PP$. In particular $\mathfrak{S} (h)$ and $B(p,h)$ are real non-negative numbers,
\begin{equation}\label{Wri12}
\mathfrak{S} (h),\,B(p,h)\in[0,\infty)\qquad \text{ for any } h\in\G_0,\,p\in\PP.
\end{equation}

We would like to show now that $\mathfrak{S} (h)\gtrsim_r 1$ for a large set of elements $h\in\G_0$, in order to be able to exploit the expansion \eqref{id:asym}. We notice first that for any integer $r$ sufficiently large there is $p_0(r)\in\PP$ such that
\begin{align}\label{Wri15} 
1/2 \le\prod_{p \in \PP,\,p \ge p_0(r)} B(p,h) \le 3/2,
\end{align}
for any $h\in\G_0$, due to the rapid decay of the coefficients $G(a/q)$ in \eqref{estGauss}. Moreover, using the formulas \eqref{id:12} and \eqref{5.1},
\begin{equation*}
B(p,h)=\frac{M(p^n,h)}{p^{n(2r-d-d')}}+O_r(2^{-n/\delta})
\end{equation*}
for any $n\geq 1$, $h\in\G_0$ and $p\in\PP$, $p\leq p_0(r)$. In view of the definition \eqref{Wri17},
\begin{equation*}
\sum_{h\in\Z_q^{d+d'}}M(q,h)= q^{2r} \qquad\text{ for any }q\geq 1.
\end{equation*}
Therefore we can fix $n=n(r)$ sufficiently large and $h_{p^n}\in \Z_{p^n}^{d+d'}$ such that $B(p,h_{p^n})\geq 1/2$ for any prime $p\leq p_0(r)$. Therefore we can fix
\begin{equation*}
Q=Q(r)=\prod_{p\in\PP,\,p<p_0(r)}p^{n(r)},\qquad g_0\in\Z_Q^{d+d'},\qquad g_0\equiv h_{p^n}\,(\mathrm{mod}\,p^n),
\end{equation*}
with the property that $B(p,g)\ge 1/2$ for any prime $p<p_0(r)$ and $g \in g_0\HH_Q$. Thus
\begin{equation}\label{Wri18}
\mathfrak{S}(g)\gtrsim_r 1\text{ for any }g\in g_0\HH_Q.
\end{equation}

{\bf{Step 4.}} Finally we analyze the contribution of the singular integral. Since 
\begin{equation*}
\int_{\R^{d+d'}} \Phi (\zeta) \ex(-(N^{-1}\circ g) . \zeta)\, d\zeta=\int_{\R^{d+d'}} \Phi (\zeta)\, d\zeta +O_{r,g}(N^{-1}),
\end{equation*}
due to \eqref{id:9}, to prove the approximate identity \eqref{Wri2.5} it suffices to prove that 
\begin{equation}\label{Wri20}
\int_{\R^{d+d'}} \Phi (\zeta)\, d\zeta\gtrsim_r 1.
\end{equation}
We fix a smooth function $\chi:\R^{d+d'}\to[0,1]$, satisfying $\chi(x)=1$ if $|x|\leq 1/2$, $\chi(x)=0$ if $|x|\geq 2$, and $\int_{\R^{d+d'}}\chi(x)\,dx =1 $. For $\epsilon\leq\epsilon(r)$ sufficiently small we write
\begin{equation}\label{Wri21}
\int_{\R^{d+d'}} \Phi (\zeta)\widehat{\chi}(\epsilon\zeta)\, d\zeta=\int_{[-1,1]^{2r}}\epsilon^{-(d+d')}\chi(D(z,w)/\epsilon)\,dzdw,
\end{equation}
using the definition \eqref{def:Phi}. In particular, by letting $\epsilon\to 0$, $\int_{\R^{d+d'}} \Phi (\zeta)\, d\zeta$ is a real non-negative number. Moreover, the lower bound \eqref{Wri20} follows from \eqref{Wri21} provided that we can show that there is a point $(z_0,w_0)\in[-1,1]^{2r}$ such that
\begin{equation}\label{Wri22}
D(z_0,w_0)=0\qquad\text{ and }
\qquad \mathrm{rank}[\nabla_{\! z,w}D(z_0,w_0)]=d+d'.
\end{equation}
We notice that this follows easily from Lemma \ref{L4.2} below. 
\end{proof}

\begin{lemma}\label{L4.2} Let $r\geq r_0(d)$. Then there exists $(n,m)\in \Z^{2r}$ such that 
\begin{equation}\label{Wri23} \mathrm{rank}[\nabla_{\! x,y}D(n,m)]=d+d'.
\end{equation}
\end{lemma}

Indeed, writing $D_r(x,y)=D(x,y):\R^{2r}\to \mathbb{G}_0^\#$, we have that 
$D_{r}(x,y) \cdot D_{r}(n,m)^{-1} = D_{2r}((x,m'),(y,n'))$ with $n'=(n_r,\ldots,n_1),\ m'=(m_r,\ldots,m_1)$. Assuming \eqref{Wri23} it is clear that the map $D_{2r}((x,m'),(y,n'))$ has maximal rank at $z_0=(n,m'),\ w_0=(m,n')$ and \eqref{Wri22} follows. 

The proof of Lemma \ref{L4.2} is based on counting points $(n,m)\in [N]^{2r}$ at which the rank of the map $\nabla_{\! x,y}D$ drops. This was also crucial in obtaining the nilpotent Weyl estimate \eqref{pro0.2}. 

\begin{proof}[Proof of Lemma~\ref{L4.2}] Let $N$ be sufficiently large with respect to $r,d$. It is enough to show that there is a constant $C_d >0$ ($C_d = 2d(d+d')^2$ works here) such that
\begin{equation}\label{Wri24} 
|\{n\in [N]^r: \mathrm{rank}[\nabla_{\! x} D(n,m)]<d+d'\}|
\ls_{d,r} 
N^{(r+1)/2 + C_d},
\end{equation}
holds uniformly for $m\in [N]^r$. Fix $m\in [N]^r$. If 
$\ \mathrm{rank}[\nabla_{\! x} D(n,m)]<d+d'$ then by Cramer's rule there exists $b_{l_1 l_2}\in \Z$, $|b_{l_1 l_2}|\ls N^{(2d-1)(d+d')}$ with $b_{l_1l_2}\neq 0$ for at least one $0\leq l_2<l_1\leq d$, such that
\begin{equation}\label{Wri25}
\sum_{0\leq l_2<l_1\leq d} b_{l_1 l_2}\, [\partial_j D(n,m)]_{l_1 l_2}=0 \quad\quad \textit{for all}\quad 1\leq j\leq r.
\end{equation}
From \eqref{pro0.4} we have that $[\partial_j D(n,m)]_{l_1 0}= - l_1 n_j^{l_1-1}$, and for $1\leq l_2$,
\begin{align}\label{Wri26}
[\partial_j D(n,m)]_{l_1 l_2} 
& = 
l_1 n_j^{l_1-1} \sum_{k>j} (n_k^{l_2}-m_k^{l_2}) 
+l_2 n_j^{l_2-1} \sum_{k<j} (n_k^{l_1}-m_k^{l_1})  \\ \nonumber
& \quad
-l_1 n_j^{l_1-1} m_j^{l_2}
+ 
(l_1 + l_2)n_j^{l_1 + l_2 - 1}.
\end{align}
We want to only include terms $k\leq j$ and to achieve that we introduce the parameters 
\[T_{l}=T_{l}(n,m)= \sum_{k=1}^r (n_k^{l}-m_k^{l}), \quad\textit{for}\quad\quad 1\leq l< d.\]

Note that $T_l \in [-2r N^{d-1},2r N^{d-1}]$. For fixed $T=(T_l)_{1\leq l<d}$, write \[\sum_{k>j} (n_k^{l_2}-m_k^{l_2}) = T_{l_2}(n,m) -\sum_{k\leq j} (n_k^{l_2}-m_k^{l_2}),\]
Substituting into \eqref{Wri26}, we obtain, up to lower degree terms in the variables $n=(n_1,\ldots,n_r)$,
\begin{equation}\label{Wri27}
[\partial_j D(n,m)]_{l_1 l_2} = 
-l_1 \sum_{k\leq j}n_j^{l_1-1} n_k^{l_2} +l_2\sum_{k<j} n_j^{l_2-1} n_k^{l_1}
+ 
(l_1 + l_2)n_j^{l_1 + l_2 - 1},  
\end{equation}
for $1\leq l_2< d$.
Thus the system in \eqref{Wri25} takes the form 
\begin{equation}\label{Wri28}
\sum_{0\leq l_2<l_1\leq d} b_{l_1 l_2}\, P^{j,T,m}_{l_1 l_2}(n_1,\ldots,n_j)=0, \quad\quad 1\leq j\leq r.
\end{equation}
Notice that for fixed $n_1,\ldots,n_{2j-2}$ with $j\leq r/2$, the left side of \eqref{Wri28} with $j$ replaced by $2j$ contains the monomials $- b_{l_1 0} l_1 n_{2j}^{l_1-1}$ and $b_{l_1 l_2} l_2 n_{2j}^{l_2 - 1} n_{2j-1}^{l_1}$, and hence is nonvanishing in the variables $n_{2j-1}, n_{2j}$. 
This, thanks to \cite[Lemma~5.3]{IMW}, implies that number of solutions to \eqref{Wri28} is at most $2(d+d')(N+1)$ in the variables $n_{2j-1}, n_{2j}$. 

As the number of choices for parameters $b=(b_{l_1 l_2})_{0\leq l_2<l_1\leq d}$ and $T=(T_l)_{1 \leq l<d}$ is $\ls_{r,d} N^{C_d}$ (with, say $C_d = 2d (d+d')^2$), \eqref{Wri23} follows.
\end{proof}

\begin{remark}
We remark that \eqref{Wri23} together with the argument proving \eqref{Wri22} also implies that the map $D_{2r}: \R^{4r}\to \mathbb{G}_0^\#$ is surjective. 
Indeed, the image of the map $D_r$ must contain an open ball $B(g,\delta) $ thus the image of $D_{2r}$ must contain an open ball $B(0,\delta')\subseteq B(g,\de) B(g,\de)^{-1}$ centered at the origin, then by homogeneity the whole space $\mathbb{G}_0^\#$.

Finally, we have in fact shown that for $r\geq r_0(d)$ the equations $D_{r}(n,m)=0$ have a non-singular integer solution and hence a non-singular $p$-adic solutions for each prime $p$. It is well-known from a general form of Hensel's Lemma, see \cite[Lemma 5.21]{Green}, that all local factors $B(p,0)$ are non-vanishing and hence the singular series $\mathfrak{S}(0)>0$.  Thus  Theorem~\ref{thm:1} and in particular asymptotic formula \eqref{Wri2.5} holds for $g=0$ with $c_r(0)>0$. This gives the precise nilpotent analogue of the asymptotic formulae for the number of solutions to the Vinogradov system: $\sum_{j=1}^r (n_j^l-m_j^l)=0$ for all $1\leq l\leq d$, see \cite{Da,Woo} for both historical and recent breakthrough developments.
\end{remark}

\end{document}